\documentclass[11pt, twoside, leqno]{article}

\usepackage{amsmath}
\usepackage{amssymb}
\usepackage{titletoc}
\usepackage{mathrsfs}
\usepackage{amsthm}
\usepackage{indentfirst}
\usepackage{color}
\usepackage{txfonts}
\usepackage{enumerate}

\allowdisplaybreaks

\newtheorem{theorem}{Theorem}[section]
\newtheorem{lemma}[theorem]{Lemma}
\newtheorem{corollary}[theorem]{Corollary}
\newtheorem{proposition}[theorem]{Proposition}

\theoremstyle{definition}
\newtheorem{remark}[theorem]{Remark}
\newtheorem{definition}[theorem]{Definition}

\numberwithin{equation}{section}

\begin{document}

\title{\bf\Large New Characterizations and Properties of Matrix $A_\infty$
Weights\footnotetext{\hspace{-0.25cm}2020 {\it Mathematics Subject Classification}.
Primary 46E30; Secondary 47A56, 15A15, 42B35.\endgraf
{\it Key words and phrases:}
matrix weight,
Muckenhoupt condition,
reverse H\"older inequality,
self-improvement.\endgraf
This project is partially supported by the National Key Research
and Development Program of China (Grant No.\
2020YFA0712900), the National
Natural Science Foundation of China (Grant Nos. 12122102, 12371093,
and 12071197), the Fundamental Research Funds
for the Central Universities (Grant No.\ 2233300008),
and the Research Council of Finland (Grant No.\ 346314).}}
\date{}
\author{Fan Bu, Tuomas Hyt\"onen\footnote{Corresponding author,
E-mail: \texttt{tuomas.hytonen@helsinki.fi}/November 10,
2023/Final version.}, \
Dachun Yang and Wen Yuan}

\maketitle

\vspace{-0.8cm}

\begin{center}
\begin{minipage}{13cm}
{\small {\bf Abstract}\quad
We provide several new characterizations of $A_{p,\infty}$-matrix weights,
originally introduced by A. Volberg
as matrix-valued substitutes of the classical $A_\infty$ weights.
In analogy with the notion of $A_p$-dimension of matrix weights
introduced in our previous work, we introduce the concepts of
the lower and the upper dimensions of $A_{p,\infty}$-matrix weights,
which enable us to obtain sharp estimates related to their reducing operators.
In a follow-up work, these results will play a key role
in the study of function spaces with $A_{p,\infty}$-matrix weights,
which extends earlier results in the more restricted class of $A_p$-matrix weights.
}
\end{minipage}
\end{center}

\vspace{0.2cm}


\vspace{0.2cm}

\section{Introduction}

For a scalar-valued weight $w$,
it is well known that its membership in the Muckenhoupt class $A_p$
is the characterising condition for several inequalities
involving the boundedness of some operators (like the maximal operator \cite{Mu72},
the Hilbert transform \cite{HMW}, or more general singular integrals \cite{CF74})
in the weighted space $L^p(w)$.
On the other hand, the weaker condition $w\in A_\infty$
is often enough for various other estimates in these spaces
(like the comparison of singular integrals with the maximal operators \cite{CF74}).

The study of the space $L^2(W)$ with a matrix weight $W$ goes back to
Wiener and Masani \cite[\S 4]{wm58}.
Their work focused on the prediction theory for multivariate stochastic processes.
In order to address the problems related to
multivariate random stationary process and Toeplitz operators,
Treil and Volberg \cite{tv97} introduced $A_2$-matrix weights
and proved that the Hilbert transform is bounded on $L^2(W)$ if and only if $W\in A_2$.
Extensions to $L^p(W)$ with $W\in A_p$ for general $p\in(1,\infty)$
were later found by Nazarov and Treil \cite{nt96} and Volberg \cite{v97}.
For many other results on $A_p$-matrix weights and
$L^p(W)$, we refer the reader to \cite{b01,bc22,cg01,cim18,cimpr,cmr16,DKPS,dly21,g03,HPV,im19,ipr21,MRR22,nptv}.

In the same work already mentioned, Volberg \cite{v97} also introduced
an analogue of the $A_\infty$ condition for matrix weights.
To be precise, for matrix-valued weights, this condition splits into
a family of conditions $A_{p,\infty}$, $p\in(0,\infty)$, but they all reduce to
the classical $A_\infty$, when the weight is scalar-valued.
In analogy with results in the scalar case, it was observed in \cite{v97} that the equivalence of certain matrix-weighted Triebel--Lizorkin norms remains valid for this larger class of weights.

Our interest in $A_{p,\infty}$ weights is motivated by a more general class of similar applications in the theory of function spaces with matrix weights.
Following the development of matrix-weighted Besov spaces $\dot B^s_{p,q}(W)$ by
Frazier and Roudenko \cite{fr04,fr08,ro03,ro04} and Bu et al. \cite{byy},
and matrix-weighted Triebel--Lizorkin spaces $\dot F^s_{p,q}(W)$
by Frazier and Roudenko \cite{fr21} and Wang et al. \cite{wyz}
(after some results in the special case $\dot F^{0}_{p,2}(W)$
already in \cite{nt96,v97} and more recently in \cite{isr21}),
the present authors \cite{bhyy} recently introduced and investigated
a rather general class of Besov-type and Triebel--Lizorkin-type spaces with matrix weights.
However, the theory of \cite{bhyy} was developed for $A_p$-matrix weights only.
Extending these results to the broader generality of $A_{p,\infty}$-matrix weights
motivated our systematic study of this interesting class of weights in this article.
Applications to Besov-type and Triebel--Lizorkin-type spaces
with $A_{p,\infty}$-matrix weights will be considered in a follow-up work,
but the present results should also have an independent interest
and many other applications involving matrix-valued extensions of
classical results dealing with $A_\infty$ weights.

The organization of the remainder of this article is as follows.
In Section \ref{sec:background},
we recall some basic concepts and results from the theory of matrix weights.
In Section \ref{sec:matrixAinfty},
we introduce a new definition of the matrix $A_{p,\infty}$ class
and prove its equivalence with Volberg's definition,
together with several further characterisations.
In Section \ref{sec:self-improvement},
we explore a certain self-improvement property of the original definition
and obtain a matrix-valued substitute of
the classical identity $A_\infty=\bigcup_{q\in[1,\infty)}A_q$.
Section \ref{sec:RHI} is concerned with reverse H\"older inequalities,
which are another manifestation of self-improvement.
In Section \ref{sec:dimApInfty},
we introduce and study the concepts of the upper
and the lower dimensions of $A_{p,\infty}$ weights
and use them to estimate norms of the form $\|A_Q A_R^{-1}\|$,
where $A_Q$ and $A_R$ are the \emph{reducing operators}
related to $W\in A_{p,\infty}$ on different cubes $Q,R\subset\mathbb R^n$.
Such estimates for $A_p$-matrix weights played a key role
in the development of the theory of matrix-weighted Triebel--Lizorkin spaces
in \cite{bhyy,fr21}, and will be extensively used in our follow-up work,
where we extend this theory to the larger class of $A_{p,\infty}$-matrix weights.
In the final Section \ref{sec:examples},
we illustrate the new concepts by several examples,
which also show the sharpness of the results obtained in Section \ref{sec:dimApInfty}.

At the end of this section, we make some conventions on notation.
Through the whole article, we work on $\mathbb R^n$.
The \emph{ball} $B$ of $\mathbb{R}^n$,
centered at $x\in\mathbb{R}^n$ with radius $r\in(0,\infty)$,
is defined by setting
$$
B:=\{y\in\mathbb{R}^n:\ |x-y|<r\}=:B(x,r);
$$
moreover, for any $\lambda\in(0,\infty)$, $\lambda B:=B(x,\lambda r)$.
A \emph{cube} $Q$ of $\mathbb{R}^n$ always has finite edge length
and all edges of cubes are always assumed to be parallel to coordinate axes,
but $Q$ is not required to be open or closed.
For any cube $Q$ of $\mathbb{R}^n$,
let $c_Q$ be its center and $\ell(Q)$ its edge length.
For any $\lambda\in(0,\infty)$ and any cube $Q$ of $\mathbb{R}^n$,
let $\lambda Q$ be the cube with the same center of $Q$ and the edge length $\lambda\ell(Q)$.
For any $r\in\mathbb{R}$, $r_+$ is defined as $r_+:=\max\{0,r\}$
and $r_-$ is defined as $r_-:=\max\{0,-r\}$.
For any $t\in(0,\infty)$, $\log_+t:=\max\{0,\log t\}$.
For any $a,b\in\mathbb{R}$, $a\wedge b:=\min\{a,b\}$ and $a\vee b:=\max\{a,b\}$.
The symbol $C$ denotes a positive constant which is independent
of the main parameters involved, but may vary from line to line.
The symbol $A\lesssim B$ means that $A\leq CB$ for some positive constant $C$,
while $A\sim B$ means $A\lesssim B\lesssim A$.
Let $\mathbb N:=\{1,2,\ldots\}$, $\mathbb Z_+:=\mathbb N\cup\{0\}$, and $\mathbb Z_+^n:=(\mathbb Z_+)^n$.
We use $\mathbf{0}$ to denote the \emph{origin} of $\mathbb{R}^n$.
For any set $E\subset\mathbb{R}^n$,
we use $\mathbf 1_E$ to denote its \emph{characteristic function}.
For any $p\in(0,\infty]$, the \emph{Lebesgue space} $L^p(\mathbb{R}^n)$
has the usual meaning, and the \emph{local Lebesgue space}
$L^p_{\mathrm{loc}}(\mathbb{R}^n)$ is defined to be the set of
all measurable functions $f$ on $\mathbb{R}^n$ such that,
for any bounded measurable set $E$,
$$
\|f\|_{L^p(E)}:=\|f\mathbf{1}_E\|_{L^p(\mathbb{R}^n)}<\infty.
$$
In what follows, we denote $L^p(\mathbb{R}^n)$ and $L^p_{\mathrm{loc}}(\mathbb{R}^n)$
simply, respectively, by $L^p$ and $L^p_{\mathrm{loc}}$.
For any measurable function $w$ on $\mathbb{R}^n$
and any measurable set $E\subset\mathbb{R}^n$, let
$$
w(E):=\int_Ew(x)\,dx
$$
and, for any measurable set $E\subset\mathbb{R}^n$ with $|E|\in(0,\infty)$, let
$$
\fint_E w(x)\,dx:=\frac{1}{|E|}\int_E w(x)\,dx.
$$
The \emph{Hardy--Littlewood maximal operator} $\mathcal{M}$ is defined by setting,
for any $f\in L^1_{\mathrm{loc}}(\mathbb{R}^n)$ and $x\in\mathbb{R}^n$,
\begin{equation}\label{maximal}
\mathcal{M}(f)(x):=\sup_{\mathrm{ball}\,B\ni x}\fint_B|f(y)|\,dy.
\end{equation}
Also, when we prove a proposition (or the like),
in its proof we always use the same symbols as in
the statement itself of that proposition (or the like).

\section{Background on matrix weights}\label{sec:background}

In this section, we recall some basic concepts and results
from the theory of matrix weights.
Let us begin with some basic concepts on matrices.

For any $m,n\in\mathbb{N}$,
the set of all $m\times n$ complex-valued matrices is denoted by $M_{m,n}(\mathbb{C})$
and $M_{m,m}(\mathbb{C})$ is simply denoted by $M_{m}(\mathbb{C})$.
The zero matrix in $M_{m,n}(\mathbb{C})$ is denoted by $O_{m,n}$
and $O_{m,m}$ is simply denoted by $O_m$.
For any $A:=[a_{ij}]\in M_{m,n}(\mathbb{C})$,
the conjugate transpose of $A$ is denoted by $A^*$.
For any $A\in M_m(\mathbb{C})$, let
\begin{equation}\label{n1}
\|A\|:=\sup_{\vec z\in\mathbb{C}^m,\,|\vec z|=1}|A\vec z|.
\end{equation}

In what follows, we regard $\mathbb{C}^m$ as $M_{m,1}(\mathbb{C})$
and let $\vec{\mathbf{0}}:=(0,\ldots,0)^\mathrm{T}\in\mathbb{C}^m$.
The matrix $A\in M_m(\mathbb C)$ is called a \emph{Hermitian matrix} if $A^*=A$ and
is called a \emph{unitary matrix} if $A^*A=I_m$,
where $I_m$ is the identity matrix.
For any $A:=[a_{ij}]\in M_{m,n}(\mathbb{C})$
and $\vec z:=(z_1,\ldots,z_m)^T\in\mathbb{C}^m$,
\begin{align*}
[A\ \vec z]:=\begin{bmatrix}
a_{11}&\cdots&a_{1,n}&z_1\\
\vdots&\ddots&\vdots&\vdots\\
a_{m1}&\cdots&a_{m,n}&z_m
\end{bmatrix}.
\end{align*}

Now, we recall the concepts of positive definite matrices
and nonnegative definite matrices (see, for instance, \cite[(7.1.1a) and (7.1.1b)]{hj13}).

\begin{definition}
A matrix $A\in M_m(\mathbb{C})$ is said to be \emph{positive definite}
if, for any $\vec z\in\mathbb{C}^m\setminus\{\vec{\mathbf{0}}\}$, $\vec z^*A\vec z>0$,
and $A$ is said to be \emph{nonnegative definite} if,
for any $\vec z\in\mathbb{C}^m$, $\vec z^*A\vec z\geq0$.
\end{definition}

\begin{remark}\label{1001}
It is well known that any nonnegative definite matrix is always Hermitian
(see, for instance, \cite[Theorem 4.1.4]{hj13}).
\end{remark}

From Remark \ref{1001} and \cite[Theorem 5.6.2(d)]{hj13},
we immediately deduce the following conclusion; we omit the details.

\begin{lemma}\label{exchange}
Let $A,B\in M_m(\mathbb{C})$ be two nonnegative definite matrices.
Then $\|AB\|=\|BA\|$ with the same norm $\|\cdot\|$ as in \eqref{n1}.
\end{lemma}

Let $A\in M_m(\mathbb{C})$ be a positive definite matrix
and have eigenvalues $\{\lambda_i\}_{i=1}^m$.
Due to \cite[Theorem 2.5.6(c)]{hj13},
we find that there exists a unitary matrix $U\in M_m(\mathbb{C})$ such that
\begin{equation}\label{500}
A=U\operatorname{diag}\,(\lambda_1,\ldots,\lambda_m)U^*.
\end{equation}
Moreover, by \cite[Theorem 4.1.8]{hj13},
we find $\{\lambda_i\}_{i=1}^m\subset(0,\infty)$.
The following definition is based on these conclusions
and can be found in \cite[(6.2.1)]{hj94} (see also \cite[Definition 1.2]{h08}).

\begin{definition}
Let $A\in M_m(\mathbb{C})$ be a positive definite matrix
and have positive eigenvalues $\{\lambda_i\}_{i=1}^m$.
For any $\alpha\in\mathbb{R}$, define
$$
A^\alpha:=U\operatorname{diag}\left(\lambda_1^\alpha,\ldots,\lambda_m^\alpha\right)U^*,
$$
where $U$ is the same as in \eqref{500}.
\end{definition}

\begin{remark}
From \cite[p.\,408]{hj94}, we deduce that $A^\alpha$
is independent of the choices of the order of $\{\lambda_i\}_{i=1}^m$ and $U$,
and hence $A^\alpha$ is well defined.
\end{remark}

Now, we recall the concept of matrix weights (see, for instance, \cite{nt96,tv97,v97}).

\begin{definition}\label{MatrixWeight}
A matrix-valued function $W:\ \mathbb{R}^n\to M_m(\mathbb{C})$ is called
a \emph{matrix weight} if $W$ satisfies that
\begin{enumerate}[\rm(i)]
\item for any $x\in\mathbb{R}^n$, $W(x)$ is nonnegative definite;
\item for almost every $x\in\mathbb{R}^n$, $W(x)$ is invertible;
\item the entries of $W$ are all locally integrable.
\end{enumerate}
A \emph{scalar weight} is a matrix weight with $m=1$.
\end{definition}

Next, we recall the concept of reducing operators (see, for instance, \cite[(3.1)]{v97}).

\begin{definition}\label{reduce}
Let $p\in(0,\infty)$, $W$ be a matrix weight,
and $E\subset\mathbb{R}^n$ a bounded measurable set satisfying $|E|\in(0,\infty)$.
The matrix $A_E\in M_m(\mathbb{C})$ is called a \emph{reducing operator} of order $p$ for $W$
if $A_E$ is positive definite and,
for any $\vec z\in\mathbb{C}^m$,
\begin{equation}\label{equ_reduce}
\left|A_E\vec z\right|
\sim\left[\fint_E\left|W^{\frac{1}{p}}(x)\vec z\right|^p\,dx\right]^{\frac{1}{p}},
\end{equation}
where the positive equivalence constants depend only on $m$ and $p$.
\end{definition}

\begin{remark}
In Definition \ref{reduce}, the existence of $A_E$ is guaranteed by
\cite[Proposition 1.2]{g03} and \cite[p.\,1237]{fr04}; we omit the details.
\end{remark}

The following lemma is just \cite[Lemma 2.11]{bhyy}.

\begin{lemma}\label{reduceM}
Let $p\in(0,\infty)$, $W$ be a matrix weight,
and $E\subset\mathbb{R}^n$ a bounded measurable set satisfying $|E|\in(0,\infty)$.
If $A_E$ is a reducing operator of order $p$ for $W$,
then, for any matrix $M\in M_m(\mathbb{C})$,
\begin{align*}
\|A_EM\|\sim\left[\fint_E\left\|W^{\frac{1}{p}}(x)M\right\|^p\,dx\right]^{\frac{1}{p}},
\end{align*}
where the positive equivalence constants depend only on $m$ and $p$.
\end{lemma}

Corresponding to the classical weight class $A_p(\mathbb R^n)$
(see, for instance, \cite[Definitions 7.1.1 and 7.1.3]{g14c}),
we have the following concept of $A_p$-matrix weights
(see, for instance, \cite[p.\,490]{fr21}).

\begin{definition}\label{def ap}
Let $p\in(0,\infty)$. A matrix weight $W$ on $\mathbb{R}^n$
is called an $A_p(\mathbb{R}^n,\mathbb{C}^m)$-\emph{matrix weight}
if $W$ satisfies that, when $p\in(0,1]$,
\begin{align*}
[W]_{A_p(\mathbb{R}^n,\mathbb{C}^m)}
:=\sup_{\mathrm{cube}\,Q}\mathop{\mathrm{\,ess\,sup\,}}_{y\in Q}
\fint_Q\left\|W^{\frac{1}{p}}(x)W^{-\frac{1}{p}}(y)\right\|^p\,dx
<\infty
\end{align*}
or that, when $p\in(1,\infty)$,
\begin{align*}
[W]_{A_p(\mathbb{R}^n,\mathbb{C}^m)}
:=\sup_{\mathrm{cube}\,Q}
\fint_Q\left[\fint_Q\left\|W^{\frac{1}{p}}(x)W^{-\frac{1}{p}}(y)\right\|^{p'}
\,dy\right]^{\frac{p}{p'}}\,dx
<\infty,
\end{align*}
where $\frac{1}{p}+\frac{1}{p'}=1$.
The $A_p(\mathbb{R}^n,\mathbb{C}^m)$-matrix weights
reduce to $A_p(\mathbb R^n)$-weights when $m=1$.
\end{definition}

If there is no confusion,
we will denote $A_p(\mathbb{R}^n,\mathbb{C}^m)$ simply by $A_p$.
The following lemma is contained in \cite[Lemma 2.15 and Corollary 2.16]{bhyy}.

\begin{lemma}\label{Ap dual}
Let $p\in(1,\infty)$, $\frac{1}{p}+\frac{1}{p'}=1$, and matrix weight $W\in A_p$.
Then $\widetilde W:=W^{-\frac{p'}{p}}\in A_{p'}$ and
\begin{align}\label{dual matrix weight}
[W]_{A_p}\sim\left[W^{-\frac{p'}{p}}\right]_{A_{p'}}^{\frac{p}{p'}},
\end{align}
where the positive equivalence constants depend only on $m$ and $p$.
Let $Q$ be a cube and $A_Q$ and $\widetilde A_Q$, respectively,
the reducing operators of order $p$ for $W$ and of order $p'$ for $\widetilde W$.
Then, for any $M\in M_m(\mathbb C)$,
\begin{equation}\label{AQ inv}
\left\|A_Q^{-1}M\right\|
\sim\left\|\widetilde A_QM\right\|
\sim\left[\fint_Q\left\|W^{-\frac{1}{p}}(x)M\right\|^{p'}\,dx\right]^{\frac{1}{p'}},
\end{equation}
where the positive equivalence constants
depend only on $m$, $p$, and $[W]_{A_p}$.
\end{lemma}

\section{Matrix $A_\infty$ weights}\label{sec:matrixAinfty}

Corresponding to $A_\infty(\mathbb R^n)$
(see, for instance, \cite[Definitions 7.3.1]{g14c}),
we introduce the concept of $A_{p,\infty}$-matrix weights as follows.
A priori, this is different from the definition given by Volberg \cite{v97},
but we will shortly prove that these definitions
(together with several further characterizations) are actually equivalent.

\begin{definition}\label{def ap,infty}
Let $p\in(0,\infty)$. A matrix weight $W$ on $\mathbb{R}^n$
is called an $A_{p,\infty}(\mathbb{R}^n,\mathbb{C}^m)$-\emph{matrix weight}
if $W$ satisfies that, for any cube $Q\subset\mathbb{R}^n$,
\begin{align*}
\log_+\left(\fint_Q\left\|W^{\frac{1}{p}}(x)W^{-\frac{1}{p}}(\cdot)\right\|^p\,dx\right)\in L^1(Q)
\end{align*}
and
\begin{align*}
[W]_{A_{p,\infty}}
&:=[W]_{A_{p,\infty}(\mathbb{R}^n,\mathbb{C}^m)}\\
&:=\sup_{\mathrm{cube}\,Q}\exp\left(\fint_Q\log\left(\fint_Q
\left\|W^{\frac{1}{p}}(x)W^{-\frac{1}{p}}(y)\right\|^p\,dx\right)\,dy\right)
<\infty.
\end{align*}
The $A_{p,\infty}(\mathbb{R}^n,\mathbb{C}^m)$-matrix weights
go back to $A_\infty(\mathbb R^n)$-weights when $m=1$.
\end{definition}

If there is no confusion,
we will denote $A_{p,\infty}(\mathbb{R}^n,\mathbb{C}^m)$ simply by $A_{p,\infty}$.
The following lemma gives the relations between scalar and matrix weights,
which follows immediately from their definitions; we omit the details.

\begin{lemma}\label{w vs W}
Let $p\in(0,\infty)$.
Let $w$ be a scalar weight and consider the matrix weight $W:=wI_m$,
where $I_m$ is the identity matrix. Then
\begin{enumerate}[\rm(i)]
\item $W\in A_p(\mathbb{R}^n,\mathbb{C}^m)$ if and only if $w\in A_{\max\{1,p\}}(\mathbb R^n)$;
\item $W\in A_{p,\infty}(\mathbb{R}^n,\mathbb{C}^m)$ if and only if $w\in A_{\infty}(\mathbb R^n)$.
\end{enumerate}
\end{lemma}

We recall a well-known example of scalar weights (see, for instance, \cite[Example 7.1.7]{g14c}).

\begin{lemma}\label{example Ap}
For any $a\in\mathbb{R}$ and $x\in\mathbb{R}^n$, define the scalar weight $w_a(x):=|x|^a$.
Then the following three statements hold.
\begin{enumerate}[{\rm(i)}]
\item $w_a\in A_\infty(\mathbb R^n)$ if and only if $a\in(-n,\infty)$;
\item $w_a\in A_1(\mathbb R^n)$ if and only if $a\in(-n,0]$;
\item For any $p\in(1,\infty)$,
$w_a\in A_p(\mathbb R^n)$ if and only if $a\in(-n,n(p-1))$.
\end{enumerate}
\end{lemma}

\begin{remark}
Lemma \ref{Ap dual} does not hold for $A_{p,\infty}$.
Indeed, let $p\in(1,\infty)$ and $\frac{1}{p}+\frac{1}{p'}=1$.
By Lemma \ref{example Ap}(i), we conclude for scalar weights that
\begin{align*}
w(\cdot):=|\cdot|^{n\frac{p}{p'}}\in A_\infty(\mathbb R^n)
\text{ but } w^{-\frac{p'}{p}}(\cdot)=|\cdot|^{-n}\notin A_\infty(\mathbb R^n).
\end{align*}
Combined with Lemma \ref{w vs W}, this implies for matrix weights that
\begin{align*}
W:=wI_m\in A_{p,\infty}(\mathbb{R}^n,\mathbb{C}^m)
\text{ but } W^{-\frac{p'}{p}}=w^{-\frac{p'}{p}}I_m
\notin A_{p',\infty}(\mathbb{R}^n,\mathbb{C}^m).
\end{align*}
This defect leads to an essential difficulty that
$\|A_Q^{-1}M\|$ does not have a good enough equivalent form like \eqref{AQ inv}
and hence a series of related estimates no longer hold.
\end{remark}

\begin{lemma}\label{reduceM inv}
Let $p\in(0,\infty)$, $W\in A_{p,\infty}$,
and $M\in M_m(\mathbb C)$ be nonzero.
Then, for any cube $Q\subset\mathbb{R}^n$,
\begin{align}\label{reduceM inv equ1}
\log\left[\left\|W^{-\frac{1}{p}}(\cdot)M\right\|^p\right]\in L^1(Q)
\end{align}
and, for any cube $Q\subset\mathbb{R}^n$,
\begin{align}\label{reduceM inv equ2}
\left\|A_Q^{-1}M\right\|^p
\sim\exp\left\{\fint_Q\log\left[\left\|W^{-\frac{1}{p}}(x)M\right\|^p\right]\,dx\right\},
\end{align}
where the positive equivalence constants
depend only on $m$, $p$, and $[W]_{A_{p,\infty}}$.
\end{lemma}

\begin{proof}
We first show
\begin{align}\label{reduceM inv equ4}
\log_+\left[\left\|W^{-\frac{1}{p}}(\cdot)M\right\|^p\right]\in L^1(Q).
\end{align}
By $W\in A_{p,\infty}$ and Lemma \ref{reduceM},
we find that, for any cube $Q\subset\mathbb{R}^n$,
\begin{align*}
\log_+\left[\left\|A_QW^{-\frac{1}{p}}(\cdot)\right\|^p\right]\in L^1(Q).
\end{align*}
From this and Lemma \ref{exchange}, we infer that,
for any cube $Q\subset\mathbb{R}^n$,
\begin{align*}
\log_+\left[\left\|W^{-\frac{1}{p}}(\cdot)M\right\|^p\right]
&\leq\log_+\left[\left\|W^{-\frac{1}{p}}(\cdot)A_Q\right\|^p\right]
+\log_+\left[\left\|A_Q^{-1}M\right\|^p\right]\\
&=\log_+\left[\left\|A_QW^{-\frac{1}{p}}(\cdot)\right\|^p\right]
+\log_+\left[\left\|A_Q^{-1}M\right\|^p\right]
\in L^1(Q).
\end{align*}
This finishes the proof of \eqref{reduceM inv equ4}.

Next, we prove \eqref{reduceM inv equ1} and \eqref{reduceM inv equ2}.
For any cube $Q\subset\mathbb{R}^n$ and almost every $x\in\mathbb{R}^n$,
\begin{align*}
\left\|A_Q^{-1}M\right\|^p
\leq\left\|A_Q^{-1}W^{\frac{1}{p}}(x)\right\|^p\left\|W^{-\frac{1}{p}}(x)M\right\|^p,
\end{align*}
which, together with \eqref{reduceM inv equ4}, further implies that
\begin{align}\label{2.16x}
\left\|A_Q^{-1}M\right\|^p
&\leq\exp\left\{\fint_Q\log\left[\left\|A_Q^{-1}W^{\frac{1}{p}}(x)\right\|^p\right]\,dx\right\}\\
&\quad\times\exp\left\{\fint_Q\log\left[\left\|W^{-\frac{1}{p}}(x)M\right\|^p\right]\,dx\right\}.\notag
\end{align}
By Lemma \ref{exchange}, Jensen's inequality, and Lemma \ref{reduceM}, we have,
for any cube $Q\subset\mathbb{R}^n$,
\begin{align*}
&\exp\left\{\fint_Q\log\left[\left\|A_Q^{-1}W^{\frac{1}{p}}(x)\right\|^p\right]\,dx\right\}\\
&\quad=\exp\left\{\fint_Q\log
\left[\left\|W^{\frac{1}{p}}(x)A_Q^{-1}\right\|^p\right]\,dx\right\}\notag\\
&\quad\leq\fint_Q\left\|W^{\frac{1}{p}}(x)A_Q^{-1}\right\|^p\,dx
\sim\left\|A_QA_Q^{-1}\right\|^p=1.\notag
\end{align*}
This, combined with \eqref{2.16x}, further implies that,
for any cube $Q\subset\mathbb{R}^n$,
\begin{align}\label{1}
0<\left\|A_Q^{-1}M\right\|^p
\lesssim\exp\left\{\fint_Q\log\left[\left\|W^{-\frac{1}{p}}(x)M\right\|^p\right]\,dx\right\}.
\end{align}
Using this and \eqref{reduceM inv equ4}, we obtain \eqref{reduceM inv equ1}.

On the other hand, from Lemmas \ref{exchange} and \ref{reduceM}
and the definition of $[W]_{A_{p,\infty}}$, we deduce that,
for any cube $Q\subset\mathbb{R}^n$,
\begin{align*}
&\exp\left\{\fint_Q\log\left[\left\|W^{-\frac{1}{p}}(x)M\right\|^p\right]\,dx\right\}\\
&\quad\leq\exp\left\{\fint_Q\log\left[\left\|W^{-\frac{1}{p}}(x)A_Q\right\|^p
\left\|A_Q^{-1}M\right\|^p\right]\,dx\right\}\\
&\quad=\exp\left\{\fint_Q\log\left[\left\|A_QW^{-\frac{1}{p}}(x)
\right\|^p\right]\,dx\right\}\left\|A_Q^{-1}M\right\|^p
\lesssim[W]_{A_{p,\infty}}
\left\|A_Q^{-1}M\right\|^p,
\end{align*}
which, together with \eqref{1}, further implies that \eqref{reduceM inv equ2}.
This finishes the proof of Lemma \ref{reduceM inv}.
\end{proof}

Now, we recall $\widetilde A_{p,\infty}$-matrix weights
introduced by Volberg (see \cite[(2.2)]{v97}).

\begin{definition}\label{def Volberg}
Let $p\in(0,\infty)$. A matrix weight $W$
is called an $\widetilde A_{p,\infty}(\mathbb{R}^n,\mathbb{C}^m)$-\emph{matrix weight}
if there exists a positive constant $C$ such that,
for any cube $Q\subset\mathbb{R}^n$ and any $\vec z\in\mathbb{C}^m$,
\begin{align*}
\exp\left(\fint_Q\log\left|W^{-\frac{1}{p}}(x)\vec z\right|\,dx\right)
\leq C\sup_{\vec u\in\mathbb{C}^m\setminus\{\vec{\mathbf{0}}\}}|(\vec z,\vec u)|
\left[\fint_Q\left|W^{\frac{1}{p}}(x)\vec u\right|^p\,dx\right]^{-\frac{1}{p}}.
\end{align*}
\end{definition}

In what follows, we denote $\widetilde A_{p,\infty}(\mathbb{R}^n,\mathbb{C}^m)$
simply by $\widetilde A_{p,\infty}$.
The following proposition shows that Volberg's condition is equivalent to ours;
it also gives several other equivalent variants of both conditions.

\begin{proposition}\label{compare}
Let $p\in(0,\infty)$ and $W$ be a matrix weight.
For any cube $Q\subset\mathbb R^n$, let $A_Q$ be the reducing operator of order $p$ for $W$.
Then there exists a positive constant $C$, independent of cube $Q$,
such that the following conditions are all equivalent:
\begin{enumerate}[\rm(i)]
\item\label{W1} $W\in\widetilde A_{p,\infty}$ in the sense of Definition \ref{def Volberg};
\item\label{W2}
$\displaystyle
\exp\left(\fint_Q\log\left|W^{-\frac{1}{p}}(x)\vec z\right|\,dx\right)
\leq C\left|A_Q^{-1}\vec z\right|
$
for any $\vec z\in\mathbb C^m$;
\item\label{W3}
$\displaystyle
\exp\left(\fint_Q\log_+\left|W^{-\frac{1}{p}}(x)A_Q\vec v\right|\,dx\right)\leq C
$
for any $\vec v\in\mathbb C^m$ with $|\vec v|=1$;
\item\label{W4}
$\displaystyle
\exp\left(\fint_Q\log_+\left\|W^{-\frac{1}{p}}(x)A_Q U\right\|\,dx\right)\leq C
$
for any $U\in M_m(\mathbb C)$ with $\|U\|=1$;
\item\label{W4b}
$\displaystyle
\exp\left(\fint_Q\log_+\left\|W^{-\frac{1}{p}}(x)A_Q\right\|\,dx\right)\leq C
$;
\item\label{W4c} $\displaystyle \exp\left(\fint_Q\log_+\left(\fint_Q\left\|W^{\frac1p}(x)W^{-\frac{1}{p}}(y)\right\|\,dx\right)\,dy\right)\leq C$;
\item\label{W6} $W\in A_{p,\infty}$ in the sense of Definition \ref{def ap,infty};
\item\label{W5}
$\displaystyle
\exp\left(\fint_Q\log\left\|W^{-\frac{1}{p}}(x)M\right\|\,dx\right)
\leq C\left\|A_Q^{-1}M\right\|
$
for any $M\in M_m(\mathbb C)$.
\end{enumerate}
\end{proposition}

\begin{proof}
\eqref{W1}$\Leftrightarrow$\eqref{W2}: Both conditions \eqref{W1} and \eqref{W2} involve an estimate of the same left-hand side. Concerning the right-hand sides of these estimates,
by \eqref{equ_reduce}, we conclude that,
for any cube $Q\subset\mathbb{R}^n$ and any $\vec z\in\mathbb{C}^m$,
\begin{align*}
&\sup_{\vec u\in\mathbb{C}^m\setminus\{\vec{\mathbf{0}}\}}|(\vec z,\vec u)|
\left[\fint_Q\left|W^{\frac{1}{p}}(x)\vec u\right|^p\,dx\right]^{-\frac{1}{p}}\\
&\quad\sim\sup_{\vec u\in\mathbb{C}^m\setminus\{\vec{\mathbf{0}}\}}
\frac{|(\vec z,\vec u)|}{|A_Q\vec u|}
=\sup_{\vec u\in\mathbb{C}^m\setminus\{\vec{\mathbf{0}}\}}
\frac{|(\vec z,A_Q^{-1}\vec u)|}{|\vec u|}
=\sup_{\vec u\in\mathbb{C}^m\setminus\{\vec{\mathbf{0}}\}}
\frac{|(A_Q^{-1}\vec z,\vec u)|}{|\vec u|}
=\left|A_Q^{-1}\vec z\right|.
\end{align*}
Thus, the right-hand sides of \eqref{W1} and \eqref{W2}
are also comparable, and hence the two conditions are equivalent.

\eqref{W2}$\Rightarrow$\eqref{W3}: We follow an argument contained in the proof of \cite[Lemma 3.1]{v97}.
For any cube $Q\subset\mathbb{R}^n$ and any $\vec v\in\mathbb{C}^m$ with $|\vec v|=1$,
\begin{align*}
&\fint_Q\log_+\left|W^{-\frac{1}{p}}(x)A_Q\vec v\right|\,dx\\
&\quad=\fint_Q\log\left|W^{-\frac{1}{p}}(x)A_Q\vec v\right|\,dx
+\fint_Q\log_+\left[\left|W^{-\frac{1}{p}}(x)A_Q\vec v\right|^{-1}\right]\,dx\\
&\quad=:I(\vec v)+II(\vec v).
\end{align*}
By the assumption \eqref{W2}, we obtain,
for any $\vec v\in\mathbb{C}^m$ with $|\vec v|=1$,
\begin{equation*}
I(\vec v)
\leq\log\left(C\left|A_Q^{-1}A_Q\vec v\right|\right)
=\log\left(C\left|\vec v\right|\right)
=\log C,
\end{equation*}
where $C$ is the same as in \eqref{W2}. On the other hand,
since $W^{\frac{1}{p}}(x)$ and $A_Q$ are Hermitian matrices
and by the Cauchy--Schwarz inequality, we find that,
for almost every $x\in\mathbb R^n$ and for any $\vec v\in\mathbb{C}^m$ with $|\vec v|=1$,
\begin{equation*}
1
=\left|\vec v\right|^2
=\left(W^{-\frac{1}{p}}(x)A_Q\vec v,W^{\frac{1}{p}}(x)A_Q^{-1}\vec v\right)
\leq\left|W^{-\frac{1}{p}}(x)A_Q\vec v\right|\left|W^{\frac{1}{p}}(x)A_Q^{-1}\vec v\right|.
\end{equation*}
This, combined with the elementary inequality $\log_+ t<t$ for $t\in(0,\infty)$
and \eqref{equ_reduce}, further implies that,
for any $\vec v\in\mathbb{C}^m$ with $|\vec v|=1$,
\begin{align*}
II(\vec v)
&\leq\fint_Q\log_+\left|W^{\frac{1}{p}}(x)A_Q^{-1}\vec v\right|\,dx
=\frac{1}{p}\fint_Q\log_+\left[\left|W^{\frac{1}{p}}(x)A_Q^{-1}\vec v\right|^p\right]\,dx \\
&\leq\frac{1}{p}\fint_Q\left|W^{\frac{1}{p}}(x)A_Q^{-1}\vec v\right|^p\,dx
\sim\frac{1}{p}\left|A_Q A_Q^{-1}\vec v\right|^p=\frac{1}{p}.
\end{align*}
Combining the bounds for $I(\vec v)$ and $II(\vec v)$,
we find that \eqref{W3} holds and $C_{\eqref{W3}}\lesssim C_{\eqref{W2}}$,
where the implicit positive constant depends only on $m$ and $p$.

\eqref{W3}$\Rightarrow$\eqref{W4}: For any $i\in\{1,\ldots,m\}$, let
$\vec e_i:=(0,\ldots,0,1,0,\ldots,0)^T\in\mathbb C^m$.
Then, from the Cauchy--Schwarz inequality, we infer that,
for any $M\in M_m(\mathbb C)$ and $\vec z:=(z_1,\ldots,z_m)\in\mathbb C^m$,
\begin{equation*}
\left|M\vec z\right|
\leq\sum_{i=1}^m\left|z_i\right| \left| M\vec e_i\right|
\leq\sqrt{m}\left(\sum_{i=1}^m\left|z_i\right|^2\right)^{1/2}
\max_{i\in\{1,\ldots,m\}}\left|M\vec e_i\right|
\end{equation*}
and hence
\begin{align*}
\|M\|\leq\sqrt{m}\max_{i\in\{1,\ldots,m\}}\left| M\vec e_i\right|.
\end{align*}
By this, we conclude that, for any cube $Q\subset\mathbb R^n$
and any $U\in M_m(\mathbb C)$ with $\|U\|=1$,
\begin{align}\label{estimate}
\log_+\left\|W^{-\frac{1}{p}}(x)A_Q^{-1}U\right\|
&\leq\log_+\left[\sqrt{m}\max_{i\in\{1,\ldots,m\}}\left|W^{-\frac{1}{p}}(x)A_QU\vec e_i\right|\right]\\
&\leq\log\sqrt{m}+\sum_{i=1}^m\log_+\left|W^{-\frac{1}{p}}(x)A_QU\vec e_i\right|\notag\\
&\leq\log\sqrt{m}+\sum_{i=1}^m\log_+\left|W^{-\frac{1}{p}}(x)A_Q\frac{U\vec e_i}{|U\vec e_i|}\right|,\notag
\end{align}
where the last step follows from the fact that
$|U\vec e_i|\leq\|U\|\, |\vec e_i|=1$.
(If $U\vec e_i=0$, then the $i$-th term in the sum vanishes and can simply be ignored.)
Integrating both sides of \eqref{estimate} and using the assumption \eqref{W3}
with $U\vec e_i/|U\vec e_i|$ in place of $\vec v_i$
for each $i\in\{1,\ldots,m\}$, we obtain
\begin{equation*}
\fint_Q \log_+\left\|W^{-\frac{1}{p}}(x)A_Q^{-1}U\right\|\,dx
\leq\log\sqrt{m}+\sum_{i=1}^m C=\log\sqrt{m}+Cm,
\end{equation*}
where $C$ is the same as in \eqref{W3}.
Thus, \eqref{W4} holds and $C_{\eqref{W4}}\leq\log\sqrt{m}+C_{\eqref{W3}}m$.

\eqref{W4}$\Rightarrow$\eqref{W4b}: This is obvious.

\eqref{W4b}$\Leftrightarrow$\eqref{W4c}:
Let us first change the integration variable in \eqref{W4b} to $y$.
Taking both sides of \eqref{W4b} to power $p$ and using the usual rules,
we can further replace $\|W^{-\frac1p}(y)A_Q\|$ by $\|W^{-\frac1p}(y)A_Q\|^p$.
On the other hand, we have
\begin{equation*}
\left\|W^{-\frac1p}(y)A_Q\right\|^p
=\left\|A_QW^{-\frac1p}(y)\right\|^p\sim\fint_Q \left\|W^{\frac1p}(x)W^{-\frac1p}(y)\right\|^p\,dx.
\end{equation*}
Observe that, if $A\leq cB$, where $c$ is a positive constant,
then $\log_+A\leq\log_+ c+\log_+ B$. From these three observations,
the equivalence \eqref{W4b}$\Leftrightarrow$\eqref{W4c} easily follows.

\eqref{W4c}$\Rightarrow$\eqref{W6}: This is immediate from the definition and the estimate $\log\leq\log_+$.

\eqref{W6}$\Rightarrow$\eqref{W5}:  This was proved in Lemma \ref{reduceM inv}.

\eqref{W5}$\Rightarrow$\eqref{W2}:
This follows, with $C_{\eqref{W2}}\leq C_{\eqref{W5}}$,
from applying \eqref{W5} to $M= [O_{m,m-1}\ \vec z]$.

We have thus proved that
\eqref{W1}$\Leftrightarrow$\eqref{W2}$\Rightarrow$\eqref{W3}$\Rightarrow$\eqref{W4}$\Rightarrow$\eqref{W4b}$\Leftrightarrow$\eqref{W4c}$\Rightarrow$\eqref{W6}$\Rightarrow$\eqref{W5}$\Rightarrow$\eqref{W2},
and this shows that all these conditions are equivalent,
which completes the proof of Proposition \ref{compare}.
\end{proof}

The following proposition gives an alternative way of computing the $A_{p,\infty}$ characteristic, which is sometimes useful.

\begin{proposition}\label{2.13}
Let $p\in(0,\infty)$ and $W$ be a matrix weight.
Assume, for any cube $Q\subset\mathbb R^n$,
\begin{align*}
\log_+\left(\int_Q\left\|W^{\frac{1}{p}}(x)W^{-\frac{1}{p}}(\cdot)\right\|^p\,dx\right)\in L^1(Q).
\end{align*}
Then
\begin{align}\label{the =}
[W]_{A_{p,\infty}}
&=\sup_{\mathrm{cube}\,Q}\sup_{H\in\mathcal F_Q}
\left[\fint_Q\left\|W^{\frac{1}{p}}(y)H(y)\right\|^p\,dy\right]^{-1}\\
&\quad\times\exp\left(\fint_Q\log\left(\fint_Q
\left\|W^{\frac{1}{p}}(x)H(y)\right\|^p\,dx\right)\,dy\right),\notag
\end{align}
where
\begin{align*}
\mathcal F_Q:=&\Bigg\{H:\ \mathbb{R}^n\to M_m(\mathbb{C})\text{ measurable}
:\ \fint_Q\left\|W^{\frac{1}{p}}(y)H(y)\right\|^p\,dy\neq0,\\
&\qquad\log_+\left(\fint_Q\left\|W^{\frac{1}{p}}(x)H(\cdot)\right\|^p\,dx\right)\in L^1(Q)\Bigg\}.
\end{align*}
\end{proposition}

\begin{proof}
Taking $H=W^{-\frac1p}$, we find that the right-hand side of \eqref{the =}
is at least as big as $[W]_{A_{p,\infty}}$.
Now, we show that they are equal.
By the definition of $[W]_{A_{p,\infty}}$
and Jensen's inequality, we have,
for any cube $Q\subset\mathbb{R}^n$ and any $H\in \mathcal F_Q$,
\begin{align*}
&\exp\left(\fint_Q\log\left(\fint_Q
\left\|W^{\frac{1}{p}}(x)H(y)\right\|^p\,dx\right)\,dy\right)\\
&\quad\leq\exp\left(\fint_Q\log\left(\fint_Q
\left\|W^{\frac{1}{p}}(x)W^{-\frac{1}{p}}(y)\right\|^p\,dx\right)\,dy\right)\\
&\qquad\times\exp\left\{\fint_Q\log\left[\left\|W^{\frac{1}{p}}(y)H(y)\right\|^p\right]\,dy\right\}\\
&\quad\leq[W]_{A_{p,\infty}}
\fint_Q\left\|W^{\frac{1}{p}}(y)H(y)\right\|^p\,dy.
\end{align*}
The proof of Proposition \ref{the =} is finished
after dividing both sides by the last factor
and taking supremum over all functions $H\in\mathcal F_Q$ and all cubes $Q$.
\end{proof}

As a consequence of the previous result, we obtain the following distributional estimate.

\begin{corollary}\label{8 var}
Let $p\in(0,\infty)$, $W\in A_{p,\infty}$,
and $\{A_Q\}_{\mathrm{cube}\,Q}$ be a family of
reducing operators of order $p$ for $W$.
Then there exists a positive constant $C$,
depending only on $m$ and $p$, such that,
for any cube $Q\subset\mathbb R^n$ and any $M\in(0,\infty)$,
\begin{align*}
\left|\left\{y\in Q:\ \left\|A_Q W^{-\frac1p}(y)\right\|^p\geq e^M\right\}\right|
\leq \frac{\log(C[W]_{A_{p,\infty}})}{M} |Q|.
\end{align*}
\end{corollary}

\begin{proof}
Let cube $Q\subset\mathbb R^n$ and $M\in(0,\infty)$ be fixed.
We make use of Proposition \ref{2.13} with
\begin{align*}
H:=W^{-\frac1p}\mathbf{1}_{Q\cap E_Q}
+A_Q^{-1}\mathbf{1}_{Q\setminus E_Q},
\end{align*}
where
\begin{align*}
E_Q:=\left\{y\in Q:\ \left\|A_Q W^{-\frac1p}(y)\right\|^p\geq e^M\right\}.
\end{align*}
We first prove that $H$ belongs to $\mathcal F_Q$ defined in Proposition \ref{2.13}.
By the definition of $H$, we obtain
\begin{align*}
\fint_Q\left\|W^{\frac{1}{p}}(y)H(y)\right\|^p\,dy
=\frac{1}{|Q|}\left[|E_Q|
+\int_{Q\setminus E_Q}\left\|W^{\frac1p}(y)A_Q^{-1}\right\|^p\,dy\right],
\end{align*}
which, together with Definition \ref{MatrixWeight}(ii) and Lemma \ref{reduceM}, further implies that
\begin{align}\label{8 var equ}
0<\fint_Q\left\|W^{\frac{1}{p}}(y)H(y)\right\|^p\,dy
\leq 1+\fint_Q\left\|W^{\frac{1}{p}}(y)A_Q^{-1}\right\|^p\,dy
\sim 1+\left\|A_QA_Q^{-1}\right\|^p
\sim 1.
\end{align}
Moreover, for any $y\in Q$,
\begin{align*}
\log_+\left(\fint_Q\left\|W^{\frac1p}(x)H(y)\right\|^p\,dx\right)
&=\mathbf{1}_{E_Q}(y)
\log_+\left(\fint_Q\left\|W^{\frac1p}(x)W^{-\frac1p}(y)\right\|^p\,dx\right)\\
&\quad+\mathbf{1}_{Q\setminus E_Q}(y)
\log_+\left(\fint_Q\left\|W^{\frac1p}(x)A_Q^{-1}\right\|^p\,dx\right),
\end{align*}
where the first term is in $L^1(Q)$ by the definition of $W\in A_{p,\infty}$
and the second term is clearly in the same space as an indicator function times a constant.
These, combined with \eqref{8 var equ}, further imply that $H\in\mathcal F_Q$.

From Lemma \ref{reduceM} and the definitions of $H$ and $E_Q$, we infer that
\begin{align*}
&\exp\left(\fint_Q\log\left(\fint_Q\left\|W^{\frac1p}(x)H(y)\right\|^p dx\right)\,dy\right)\\
&\quad\sim\exp\left\{\fint_Q\log\left[\|A_QH(y)\|^p\right]\,dy\right\}\\
&\quad=\exp\left(\frac{1}{|Q|}\left\{
\int_{E_Q}\log\left[\left\|A_QW^{-\frac1p}(y)\right\|^p\right]\,dy
+\int_{Q\setminus E_Q}\log\left(\left\|A_QA_Q^{-1}\right\|^p\right)\,dy\right\}\right)\\
&\quad\geq\exp\left(\frac{1}{|Q|}\int_{E_Q}M\,dy\right)
=\exp\left(\frac{|E_Q|M}{|Q|}\right),
\end{align*}
which, together with \eqref{8 var equ} and Proposition \ref{2.13}, further implies that
\begin{align*}
[W]_{A_{p,\infty}}
\gtrsim\exp\left(\frac{|E_Q|M}{|Q|}\right),
\end{align*}
from which the claimed bound follows by taking logarithms of both sides.
This finishes the proof of Corollary \ref{8 var}.
\end{proof}

The following variant is a restatement of \cite[Lemma 3.1]{v97},
there stated for matrix weights satisfying
the $\widetilde A_{p,\infty}$ condition of Definition \ref{def Volberg},
which is equivalent to $A_{p,\infty}$ by Proposition \ref{compare}.
Applying an argument similar to that used in the proof of Corollary \ref{8 var},
we could give another proof of Lemma \ref{8 prepare}, but we omit the details.

\begin{lemma}\label{8 prepare}
Let $p\in(0,\infty)$, $W\in A_{p,\infty}$,
and $\{A_Q\}_{\mathrm{cube}\,Q}$ be a family of
reducing operators of order $p$ for $W$.
Let $Q$ be a cube of $\mathbb{R}^n$ and
$\{Q_j\}_{j\in J}\subset Q$ a pairwise disjoint cube sequence.
Let $M\in(0,\infty)$.
If, for any $j\in J$,
$\|A_QA_{Q_j}^{-1}\|^p\geq e^M$,
then there exists a positive constant $C$,
depending only on $m$ and $p$, such that
\begin{align*}
\sum_{j\in J}|Q_j|\leq\frac{\log(C[W]_{A_{p,\infty}})}{M}|Q|.
\end{align*}
\end{lemma}

\section{Self-improvement properties}\label{sec:self-improvement}

The following proposition is another equivalent characterization of
$A_{p,\infty}$-matrix weights. It can be seen as a self-improvement property in the sense that the integrability of a logarithm, as in the equivalent conditions of Proposition \ref{compare}, already implies the integrability of a positive power of the same function.

\begin{proposition}\label{WAsL}
Let $p\in(0,\infty)$ and $W$ be a matrix weight.
Then $W\in A_{p,\infty}$ if and only if
there exists $u\in(0,\infty)$ such that
\begin{align}\label{WAs}
\sup_{\mathrm{cube}\,Q\subset\mathbb{R}^n}
\fint_Q\left\|W^{-\frac{1}{p}}(x)A_Q\right\|^u\,dx<\infty,
\end{align}
where $A_Q$ is the reducing operator of order $p$ for $W$.
\end{proposition}

\begin{proof}
We prove that \eqref{WAs} is equivalent to condition \eqref{W4b} of Proposition \ref{compare},
which is equivalent to $W\in A_{p,\infty}$ by the said proposition.
From the elementary inequality $\log_+ t\leq (es)^{-1}t^s$ for any $t,s\in(0,\infty)$,
it is immediate that \eqref{WAs} implies condition \eqref{W4b} of Proposition \ref{compare}.

The converse implication is more delicate
and involves an argument somewhat in the spirit of the John--Nirenberg inequality
for the exponential integrability of BMO functions.
Condition \eqref{W4b} of Proposition \ref{compare} is clearly equivalent to
\begin{align*}
M:=\sup_{\mathrm{cube}\,Q\subset\mathbb{R}^n}
\fint_Q\log_+\left\|W^{-\frac{1}{p}}(x)A_Q\right\|\,dx<\infty.
\end{align*}
Let cube $Q\subset\mathbb R^n$ be fixed.
Without loss of generality, we may assume $Q$ is a dyadic cube.
Let $\{Q_i\}_{i\in I}\subset Q$ be its maximal dyadic subcubes such that,
for any $i\in I$,
\begin{align*}
\fint_{Q_i}\log_+\left\|W^{-\frac{1}{p}}(x)A_Q\right\|\,dx
>2M.
\end{align*}
This, combined with the definition of $M$, further implies that
\begin{align}\label{WAsSum}
\sum_{i\in I}|Q_i|
&<\frac{1}{2M}\sum_{i\in I}\int_{Q_i}\log_+\left\|W^{-\frac{1}{p}}(x)A_Q\right\|\,dx\\
&\leq\frac{1}{2M}\int_{Q}\log_+\left\|W^{-\frac{1}{p}}(x)A_Q\right\|\,dx
\leq\frac{1}{2M}M|Q|=\frac{1}{2}|Q|.\notag
\end{align}
For any $i\in I$, let $\widehat Q_i:=\min\{R\in\mathscr Q:\ Q_i\subsetneqq R\}$.
By the maximality, we obtain, for any $i\in I$,
\begin{align}\label{WAsMax}
\fint_{Q_i}\log_+\left\|W^{-\frac{1}{p}}(x)A_Q\right\|\,dx
\leq 2^n\fint_{\widehat Q_i}\log_+\left\|W^{-\frac{1}{p}}(x)A_Q\right\|\,dx
\leq 2^{n+1}M.
\end{align}
We can then write
\begin{align}\label{chara1}
\mathbf{1}_Q\log_+\left\|W^{-\frac{1}{p}}A_Q\right\|
=\mathbf{1}_{Q\setminus\bigcup_i Q_i}\log_+\left\|W^{-\frac{1}{p}}A_Q\right\|
+\sum_{i\in I}\mathbf{1}_{Q_i}\log_+\left\|W^{-\frac{1}{p}}A_Q\right\|.
\end{align}
From the definition of $\{Q_i\}_{i\in I}$ and
Lebesgue's differentiation theorem
(see, for instance, \cite[Corollary 2.1.16]{g14c}), we deduce that,
for any $x\notin\bigcup_{i\in I}Q_i$,
\begin{align}\label{chara2}
\log_+\left\|W^{-\frac{1}{p}}(x)A_Q\right\|
\leq2M.
\end{align}
On the other hand, for any $i\in I$ and $x\in Q_i$, we observe that
\begin{align*}
\log_+\left\|W^{-\frac{1}{p}}(x)A_Q\right\|
\leq\log_+\left\|W^{-\frac{1}{p}}(x)A_{Q_i}\right\|
+\log_+\left\|A_{Q_i}^{-1}A_Q\right\|,
\end{align*}
where, for any $y\in Q_i$,
\begin{align*}
\log_+\left\|A_{Q_i}^{-1}A_Q\right\|
\leq \log_+\left\|A_{Q_i}^{-1}W^{\frac1p}(y)\right\|+\log_+\left\| W^{-\frac1p}(y)A_Q\right\|.
\end{align*}
Taking the average over $y\in Q_i$
and using \eqref{WAsMax} for the second term, we obtain
\begin{align*}
\log_+\left\|A_{Q_i}^{-1}A_Q\right\|
\leq\fint_{Q_i}\log_+\left\|A_{Q_i}^{-1}W^{\frac1p}(y)\right\|\,dy+2^{n+1}M.
\end{align*}
Finally, we use again the elementary inequality $\log_+ t\leq (ep)^{-1}t^p$
for any $t,p\in(0,\infty)$ and Lemmas \ref{exchange} and \ref{reduceM} to find that
\begin{align*}
\fint_{Q_i}\log_+\left\|A_{Q_i}^{-1}W^{\frac1p}(y)\right\|\,dy
\lesssim\fint_{Q_i}\left\|A_{Q_i}^{-1}W^{\frac1p}(y)\right\|^p\,dy
\sim \left\|A_{Q_i}^{-1}A_{Q_i}\right\|^p=1.
\end{align*}
Combining the previous few estimates, we have checked that
there exists a positive constant $\widetilde{M}$ such that,
for any $i\in I$ and $x\in Q_i$,
\begin{align*}
\log_+\left\|W^{-\frac{1}{p}}A_Q\right\|
\leq\log_+\left\|W^{-\frac{1}{p}}A_{Q_i}\right\|+\widetilde{M}.
\end{align*}
This, together with \eqref{chara1} and \eqref{chara2}, further implies that
\begin{align}\label{4.5x}
\mathbf{1}_{Q}\log_+\left\|W^{-\frac{1}{p}}A_Q\right\|
&\leq2M\mathbf{1}_{Q\setminus\bigcup_{i\in I}Q_i}
+\sum_{i\in I}\mathbf{1}_{Q_i}
\left(\log_+\left\|W^{-\frac{1}{p}}A_{Q_i}\right\|+\widetilde{M}\right)\\
&\leq C\mathbf{1}_{Q}
+\sum_{i\in I}\mathbf{1}_{Q_i}\log_+\left\|W^{-\frac{1}{p}}A_{Q_i}\right\|,\notag
\end{align}
where $C:=\max\{2M,\widetilde{M}\}$ and $\sum_{i\in I}|Q_i|\leq\frac12|Q|$ by \eqref{WAsSum}.

Each term on the right-hand side of \eqref{4.5x} has the same form as
the expression on the left-hand side,
and hence we can iterate this estimate. After infinitely many iterations, we arrive at
\begin{align*}
\mathbf{1}_{Q}\log_+\left\|W^{-\frac{1}{p}}A_Q\right\|
\leq C\sum_{k=0}^\infty\mathbf{1}_{\Omega_k}
=C\sum_{k=0}^\infty(k+1)\mathbf{1}_{\Omega_k\setminus\Omega_{k+1}},
\end{align*}
where $\Omega_0=Q$ and $\Omega_1=\bigcup_{i\in I}Q_i$
are the cubes constructed in the first step with $|\Omega_1|\leq\frac12|\Omega_0|$
and, in general, $\Omega_k\subset\Omega_{k-1}$ is a union of cubes
with $|\Omega_k|\leq\frac12|\Omega_{k-1}|\leq\cdots\leq 2^{-k}|Q|$.
From these, we infer that, if $u\in(0,\frac{\log 2}{C})$, then
\begin{align*}
\fint_Q\left\|W^{-\frac{1}{p}}(x)A_Q\right\|^u\,dx
&\leq\fint_Q\exp\left(u \log_+ \left\|W^{-\frac{1}{p}}(x)A_Q\right\|\right)\,dx \\
&\leq\fint_Q\sum_{k=0}^\infty e^{uC(k+1)}
\mathbf{1}_{\Omega_k\setminus\Omega_{k+1}}\,dx
\leq\sum_{k=0}^\infty e^{uC(k+1)}2^{-k}<\infty.
\end{align*}
This finishes the proof of Proposition \ref{WAsL}.
\end{proof}

Now, we can show two basic properties of
$A_p$-matrix weights and $A_{p,\infty}$-matrix weights.
Recall that, when $m=1$, we have $A_{p,\infty}(\mathbb R^n,\mathbb C)=A_\infty(\mathbb R^n)$
and $A_p(\mathbb R^n,\mathbb C)=A_{\max\{1,p\}}(\mathbb R^n)$;
hence part \eqref{Ap,infty subset union Aq} of the following proposition
recovers the classical identity
$A_\infty(\mathbb R^n)=\bigcup_{q\in[1,\infty)}A_q(\mathbb R^n)$ for scalar weights
and can be seen as a generalisation of this identity to the case of matrix weights.

\begin{proposition}\label{ap and ap,infty}
Let $p\in(0,\infty)$. Then the following two statements hold.
\begin{enumerate}[{\rm(i)}]
\item If $q\in(p,\infty)$, then $A_{p,\infty}(\mathbb R^n,\mathbb C^m)
\subset A_{q,\infty}(\mathbb R^n,\mathbb C^m)$.
Moreover, there exists a positive constant $C$,
depending only on $m$, $p$, and $q$, such that, for any matrix weight $W$,
\begin{align*}
[W]_{A_{q,\infty}(\mathbb R^n,\mathbb C^m)}
\leq C[W]_{A_{p,\infty}(\mathbb R^n,\mathbb C^m)}.
\end{align*}

\item\label{Ap,infty subset union Aq}
It holds that $A_p(\mathbb R^n,\mathbb C^m)
\subsetneqq A_{p,\infty}(\mathbb R^n,\mathbb C^m)
\subset\bigcup_{q\in(0,\infty)}A_q(\mathbb R^n,\mathbb C^m)$.
Moreover, there exists a positive constant $C$,
depending only on $m$ and $p$, such that, for any matrix weight $W$,
\begin{align*}
[W]_{A_{p,\infty}(\mathbb R^n,\mathbb C^m)}\leq C[W]_{A_p(\mathbb R^n,\mathbb C^m)},
\end{align*}
where $C=1$ when $p\in(0,1]$.
\end{enumerate}
\end{proposition}

\begin{proof}
We first prove (i).
Let us begin by recalling a useful estimate.
By the proof of \cite[Lemma 2]{ipr21}
(in which the two exponents, used in reverse roles in \cite[Lemma 2]{ipr21}
compared to the present situation, satisfy $1\leq p<q<\infty$,
but its proof still works for any $0<p<q<\infty$),
we obtain, for any $x\in\mathbb{R}^n$ and almost every $y\in\mathbb{R}^n$,
\begin{align}\label{2023.6.27}
\left\|W^{\frac{1}{q}}(x)W^{-\frac{1}{q}}(y)\right\|^q
\lesssim\left\|W^{\frac{1}{p}}(x)W^{-\frac{1}{p}}(y)\right\|^p.
\end{align}
This, combined with Definition \ref{def ap,infty}, further implies that
\begin{align*}
[W]_{A_{q,\infty}(\mathbb R^n,\mathbb C^m)}
\lesssim[W]_{A_{p,\infty}(\mathbb R^n,\mathbb C^m)}.
\end{align*}
This finishes the proof of (i).

Now, we show (ii). We first prove
$A_p(\mathbb R^n,\mathbb C^m)\subset A_{p,\infty}(\mathbb R^n,\mathbb C^m)$.
To this end, we consider the following two cases on $p$.

\emph{Case 1)} $p\in(0,1]$. In this case,
for any cube $Q\subset\mathbb{R}^n$ and almost every $y\in\mathbb{R}^n$, let
\begin{align*}
F(Q,y):=\fint_Q\left\|W^{\frac{1}{p}}(x)W^{-\frac{1}{p}}(y)\right\|^p\,dx.
\end{align*}
Then, for any cube $Q\subset\mathbb{R}^n$,
\begin{align*}
\exp\left(\fint_Q\log F(Q,y)\,dy\right)
\leq\mathop{\operatorname{ess\,sup}}_{y\in Q}F(Q,y)
\end{align*}
and hence $[W]_{A_{p,\infty}(\mathbb R^n,\mathbb C^m)}
\leq[W]_{A_p(\mathbb R^n,\mathbb C^m)}$.
This finishes the proof of $A_p(\mathbb R^n,\mathbb C^m)
\subset A_{p,\infty}(\mathbb R^n,\mathbb C^m)$ in this case.

\emph{Case 2)} $p\in(1,\infty)$. In this case, by Jensen's inequality, we conclude that
\begin{align*}
\exp\left(\fint_Q\log F(Q,y)\,dy\right)
=\left\{\exp\left(\fint_Q\log\left\{[F(Q,y)]^{\frac{p'}{p}}
\right\}\,dy\right)\right\}^{\frac{p}{p'}}
\leq\left\{\fint_Q[F(Q,y)]^{\frac{p'}{p}}\,dy\right\}^{\frac{p}{p'}},
\end{align*}
which, together with \eqref{dual matrix weight}, further implies that
\begin{align*}
[W]_{A_{p,\infty}(\mathbb R^n,\mathbb C^m)}
\leq\left[W^{-\frac{p'}{p}}\right]_{A_{p'}(\mathbb R^n,\mathbb C^m)}^{\frac{p}{p'}}
\sim[W]_{A_p(\mathbb R^n,\mathbb C^m)}.
\end{align*}
This finishes the proof of $A_p(\mathbb R^n,\mathbb C^m)
\subset A_{p,\infty}(\mathbb R^n,\mathbb C^m)$ in this case.

Next, we show $A_p(\mathbb R^n,\mathbb C^m)\neq A_{p,\infty}(\mathbb R^n,\mathbb C^m)$.
We choose $q\in(n(p-1)_+,\infty)$. From Lemma \ref{example Ap}, we deduce that
\begin{align*}
w(\cdot):=|\cdot|^q\in A_\infty(\mathbb R^n)\setminus A_{\max\{1,p\}}(\mathbb R^n).
\end{align*}
This, combined with Lemma \ref{w vs W}, further implies that
\begin{align*}
W:=wI_m\in A_{p,\infty}(\mathbb R^n,\mathbb C^m)\setminus A_p(\mathbb R^n,\mathbb C^m).
\end{align*}
This finishes the proof of $A_p(\mathbb R^n,\mathbb C^m)
\neq A_{p,\infty}(\mathbb R^n,\mathbb C^m)$.

Finally, we prove $A_{p,\infty}(\mathbb R^n,\mathbb C^m)
\subset\bigcup_{q\in(0,\infty)}A_q(\mathbb R^n,\mathbb C^m)$.
Let $W\in A_{p,\infty}(\mathbb R^n,\mathbb C^m)$.
By Proposition \ref{WAsL} and Lemmas \ref{exchange} and \ref{reduceM}, we find that
there exists $u\in(0,\infty)$ such that
\begin{align*}
\sup_{\mathrm{cube}\,Q\subset\mathbb{R}^n}
\fint_Q\left[\fint_Q\left\|W^{\frac{1}{p}}(x)W^{-\frac{1}{p}}(y)\right\|^p\,dx\right]^u\,dy<\infty.
\end{align*}
This, together with \eqref{dual matrix weight}, \eqref{2023.6.27},
and H\"older's inequality, further implies that,
for any $q\in(p,\infty)$ with $\frac{q'}q\leq u$,
\begin{align*}
[W]_{A_q(\mathbb R^n,\mathbb C^m)}
&\sim\left[W^{-\frac{q'}{q}}\right]_{A_{q'}(\mathbb R^n,\mathbb C^m)}^{\frac{q}{q'}}\\
&=\sup_{\mathrm{cube}\,Q\subset\mathbb{R}^n}
\left\{\fint_Q\left[\fint_Q\left\|W^{\frac{1}{q}}(x)W^{-\frac{1}{q}}(y)\right\|^q
\,dx\right]^{\frac{q'}q}\,dy\right\}^{\frac q{q'}}\\
&\lesssim\sup_{\mathrm{cube}\,Q\subset\mathbb{R}^n}
\left\{\fint_Q\left[\fint_Q\left\|W^{\frac{1}{p}}(x)W^{-\frac{1}{p}}(y)\right\|^p
\,dx\right]^{\frac{q'}q}\,dy\right\}^{\frac q{q'}}\\
&\leq\sup_{\mathrm{cube}\,Q\subset\mathbb{R}^n}
\left\{\fint_Q\left[\fint_Q\left\|W^{\frac{1}{p}}(x)W^{-\frac{1}{p}}(y)\right\|^p
\,dx\right]^u\,dy\right\}^{\frac1u}
<\infty
\end{align*}
and hence $W\in A_q(\mathbb R^n,\mathbb C^m)$.
This finishes the proof of $A_{p,\infty}(\mathbb R^n,\mathbb C^m)
\subset\bigcup_{q\in(0,\infty)}A_q(\mathbb R^n,\mathbb C^m)$ and hence (ii),
which completes the proof of Proposition \ref{ap and ap,infty}.
\end{proof}

\section{Reverse H\"older inequalities}\label{sec:RHI}

One of the most important properties of classical $A_\infty$ weights is the reverse H\"older inequality that goes back to \cite{CF74}.
To establish a reverse H\"older inequality for $A_{p,\infty}$-matrix weights,
we need several technical lemmas.
The following conclusion is the reverse H\"older inequality of scalar weights
(see, for instance, \cite[Theorem 2.3]{hpr12} or \cite[Theorem 2.3(a)]{hp13}).

\begin{lemma}\label{reverse Holder scalar}
Let $w$ be a scalar weight satisfying
\begin{align*}
[w]_{A_\infty(\mathbb R^n)}^*:=\sup_{\mathrm{cube}\,Q}
\frac{1}{w(Q)}\int_Q\mathcal{M}(w\mathbf{1}_Q)(x)\,dx<\infty,
\end{align*}
where $\mathcal{M}$ is the same as in \eqref{maximal}.
Then, for any
\begin{align*}
r\in\left[1,1+\frac{1}{2^{n+1}[w]_{A_\infty(\mathbb R^n)}^*-1}\right]
\end{align*}
and any cube $Q\subset\mathbb{R}^n$,
\begin{equation*}
\fint_Q[w(x)]^r\,dx
\leq 2\left[\fint_Qw(x)\,dx\right]^r.
\end{equation*}
\end{lemma}

For any scalar weight $w$, from Lebesgue's differentiation theorem
(see, for instance, \cite[Corollary 2.1.16]{g14c}), we deduce that,
for any cube $Q\subset\mathbb{R}^n$,
\begin{align*}
\int_Q\mathcal{M}(w\mathbf{1}_Q)(x)\,dx
\geq\int_Qw(x)\,dx=w(Q)
\end{align*}
and hence $[w]_{A_\infty(\mathbb R^n)}^*\in[1,\infty]$.

The following lemma is just \cite[Proposition 2.2]{hp13}.

\begin{lemma}\label{2.19}
There exists a positive constant $\widetilde C$ such that, for any scalar weight $w$,
\begin{align*}
[w]_{A_\infty(\mathbb R^n)}^*\leq \widetilde C[w]_{A_\infty(\mathbb R^n)}.
\end{align*}
\end{lemma}

Next, we establish a useful relation
between scalar weights and matrix weights.

\begin{lemma}\label{2.14}
Let $p\in(0,\infty)$ and $W\in A_{p,\infty}(\mathbb R^n,\mathbb C^m)$.
Then, for any nonzero matrix $M\in M_m(\mathbb{C})$, the scalar function
\begin{align*}
w_M:=\left\|W^{\frac{1}{p}}M\right\|^p
\end{align*}
satisfies $w_M\in A_\infty(\mathbb R^n)$ with
$[w_M]_{A_\infty(\mathbb R^n)}\leq[W]_{A_{p,\infty}(\mathbb R^n,\mathbb C^m)}$,
and
\begin{align*}
[W]_{A_{p,\infty}(\mathbb R^n,\mathbb C^m)}^{\mathrm{sc}}
:=\sup_{M\in M_m(\mathbb C)\setminus\{O_m\}}[w_M]_{A_\infty(\mathbb R^n)}^*
\leq \widetilde C[W]_{A_{p,\infty}(\mathbb R^n,\mathbb C^m)},
\end{align*}
where $\widetilde C$ is the same as in Lemma \ref{2.19}.
\end{lemma}

\begin{proof}
Let $M\in M_m(\mathbb{C})$ be nonzero.
By the definition of matrix weights, we find that $w_M$ is a scalar weight.
From Definition \ref{MatrixWeight}(ii), we infer that,
for any cube $Q\subset\mathbb{R}^n$, any $x\in Q$, and almost every $y\in Q$,
\begin{align*}
w_M(x)
\leq\left\|W^{\frac{1}{p}}(x)W^{-\frac{1}{p}}(y)\right\|^p
\left\|W^{\frac{1}{p}}(y)M\right\|^p
\end{align*}
and hence
\begin{align*}
\fint_Qw_M(x)\,dx
\leq\fint_Q\left\|W^{\frac{1}{p}}(x)W^{-\frac{1}{p}}(y)\right\|^p\,dx
\left\|W^{\frac{1}{p}}(y)M\right\|^p,
\end{align*}
which further implies that
\begin{align*}
\log\left(\fint_Qw_M(x)\,dx\right)+\log\left\{[w_M(y)]^{-1}\right\}
\leq\log\left(\fint_Q\left\|W^{\frac{1}{p}}(x)W^{-\frac{1}{p}}(y)\right\|^p\,dx\right).
\end{align*}
By this and $W\in A_{p,\infty}$,
we conclude that $\log_+(w_M^{-1})\in L^1_{\mathrm{loc}}$ and
\begin{align*}
[w_M]_{A_\infty(\mathbb R^n)}
&=\sup_{\mathrm{cube}\,Q}\fint_Qw_M(x)\,dx
\exp\left(\fint_Q\log\left\{[w_M(y)]^{-1}\right\}\,dy\right)\\
&\leq\sup_{\mathrm{cube}\,Q}\exp\left(\fint_Q\log\left(\fint_Q
\left\|W^{\frac{1}{p}}(x)W^{-\frac{1}{p}}(y)\right\|^p\,dx\right)\,dy\right)\\
&=[W]_{A_{p,\infty}(\mathbb R^n,\mathbb C^m)}.
\end{align*}
This, combined with Lemma \ref{2.19}, further implies that
$w_M\in A_\infty(\mathbb R^n)$ and
\begin{align*}
[W]_{A_{p,\infty}(\mathbb R^n,\mathbb C^m)}^{\mathrm{sc}}
\leq \widetilde C\sup_{M\in M_m(\mathbb C)\setminus\{O_m\}}[w_M]_{A_\infty(\mathbb R^n)}
\leq \widetilde C[W]_{A_{p,\infty}(\mathbb R^n,\mathbb C^m)},
\end{align*}
which completes the proof of Lemma \ref{2.14}.
\end{proof}

A definition similar in form to $[W]_{A_{p,\infty}(\mathbb R^n,\mathbb C^m)}^{\mathrm{sc}}$ can be found in
\cite[(1.5)]{nptv} or \cite[p.\,3087]{ipr21}:
For any $\vec z\in\mathbb{C}^m\setminus\{\vec{\mathbf{0}}\}$,
define the scalar weight $w_{\vec z}:=|W^{\frac{1}{p}}\vec z|^p$.
For any matrix weight $W$, define
\begin{align*}
\widetilde{[W]}_{A_{p,\infty}(\mathbb R^n,\mathbb C^m)}^{\mathrm{sc}}:=
\sup_{\vec z\in\mathbb{C}^m\setminus\{\vec{\mathbf{0}}\}}
\left[w_{\vec z}\right]_{A_\infty(\mathbb R^n)}^*.
\end{align*}
The following conclusion shows that $[W]_{A_{p,\infty}(\mathbb R^n,\mathbb C^m)}^{\mathrm{sc}}$
and $\widetilde{[W]}_{A_{p,\infty}(\mathbb R^n,\mathbb C^m)}^{\mathrm{sc}}$ are equivalent.

\begin{proposition}\label{2.18}
Let $p\in(0,\infty)$. Then there exists a positive constant $C$,
depending only on $m$ and $p$, such that, for any matrix weight $W$,
\begin{align}\label{2.18 equ}
\widetilde{[W]}_{A_{p,\infty}(\mathbb R^n,\mathbb C^m)}^{\mathrm{sc}}
\leq[W]_{A_{p,\infty}(\mathbb R^n,\mathbb C^m)}^{\mathrm{sc}}
\leq C\widetilde{[W]}_{A_{p,\infty}(\mathbb R^n,\mathbb C^m)}^{\mathrm{sc}}.
\end{align}
\end{proposition}

\begin{proof}
We first show the first inequality of \eqref{2.18 equ}.
Let $\vec z\in\mathbb{C}^m\setminus\{\vec{\mathbf{0}}\}$ and
$M_{\vec z}:=[O_{m,m-1}\ \vec z]$.
Then $M_{\vec z}\in M_m(\mathbb{C})$ is nonzero and
from \eqref{n1}, we deduce that, for any $x\in\mathbb R^n$,
\begin{align*}
w_{M_{\vec z}}(x)
=\left\|\left[O_{m,m-1}\ W^{\frac{1}{p}}(x)\vec z\right]\right\|^p
=\left|W^{\frac{1}{p}}(x)\vec z\right|^p
=w_{\vec z}(x).
\end{align*}
This, together with the definition of
$[W]_{A_{p,\infty}(\mathbb R^n,\mathbb C^m)}^{\mathrm{sc}}$, further implies that
\begin{align*}
\left[w_{\vec z}\right]_{A_\infty(\mathbb R^n)}^*
=\left[w_{M_{\vec z}}\right]_{A_\infty(\mathbb R^n)}^*
\leq[W]_{A_{p,\infty}(\mathbb R^n,\mathbb C^m)}^{\mathrm{sc}},
\end{align*}
which completes the proof of the first inequality of \eqref{2.18 equ}.

Next, we prove the second inequality of \eqref{2.18 equ}.
Let $M\in M_m(\mathbb{C})$ be nonzero.
Let $\{\vec e_i\}_{i=1}^m$ be any orthonormal basis of $\mathbb C^m$.
From [86, Lemma 3.2], we infer that, for any matrix $A\in M_m(\mathbb{C})$,
\begin{align*}
\|A\|\sim\left(\sum_{i=1}^m\left|A\vec e_i\right|^p\right)^{\frac{1}{p}}.
\end{align*}
This, combined with the definitions of $[w]_{A_p(\mathbb R^n)}^*$
and $\widetilde{[W]}_{A_{p,\infty}(\mathbb R^n,\mathbb C^m)}^{\mathrm{sc}}$,
further implies that, for any cube $Q\subset\mathbb{R}^n$,
\begin{align*}
w_M(Q)
\sim\int_Q\sum_{i=1}^mw_{M\vec e_i}(x)\,dx
=\sum_{i=1}^mw_{M\vec e_i}(Q)
\end{align*}
and
\begin{align*}
\int_Q\mathcal{M}(w_M\mathbf{1}_Q)(x)\,dx
&\sim\int_Q\mathcal{M}\left(\sum_{i=1}^m w_{M\vec e_i}\mathbf{1}_Q\right)(x)\,dx\\
&\leq\int_Q\sum_{i=1}^m\mathcal{M}(w_{M\vec e_i}\mathbf{1}_Q)(x)\,dx\\
&\leq\sum_{i=1}^m\left[w_{M\vec e_i}\right]_{A_\infty(\mathbb R^n)}^*w_{M\vec e_i}(Q)\\
&\leq\widetilde{[W]}_{A_{p,\infty}(\mathbb R^n,\mathbb C^m)}^{\mathrm{sc}}
\sum_{i=1}^mw_{M\vec e_i}(Q),
\end{align*}
where $[w_{M\vec e_i}]_{A_\infty(\mathbb R^n)}^*:=0$ if $M\vec e_i=\vec{\mathbf{0}}$.
Therefore, $[w_M]_{A_\infty(\mathbb R^n)}^*
\lesssim\widetilde{[W]}_{A_{p,\infty}(\mathbb R^n,\mathbb C^m)}^{\mathrm{sc}}$.
This finishes the proof of Proposition \ref{2.18}.
\end{proof}

The following conclusion is a simple application of Lemma \ref{2.14}.

\begin{proposition}\label{constant}
Let $p\in(0,\infty)$. Then
\begin{enumerate}[\rm(i)]
\item for any $W\in A_{p,\infty}(\mathbb R^n,\mathbb C^m)$, one has
$[W]_{A_{p,\infty}(\mathbb R^n,\mathbb C^m)}\in[1,\infty)$;
\item for any $W\in A_p(\mathbb R^n,\mathbb C^m)$, one has
$[W]_{A_p(\mathbb R^n,\mathbb C^m)}\in[1,\infty)$.
\end{enumerate}
\end{proposition}

\begin{proof}
Let $M\in M_m(\mathbb C)\setminus\{O_m\}$.
Then, by Lemma \ref{2.14} and \cite[Proposition 7.3.2(4)]{g14c}, we find that,
for any $W\in A_{p,\infty}(\mathbb R^n,\mathbb C^m)$,
\begin{align*}
[W]_{A_{p,\infty}(\mathbb R^n,\mathbb C^m)}
\geq[w_M]_{A_\infty(\mathbb R^n)}
\geq1.
\end{align*}
This finishes the proof of (i).

Now, we show (ii).
Let $W\in A_p(\mathbb R^n,\mathbb C^m)$.
Using (i) and Proposition \ref{ap and ap,infty}(ii), we obtain, if $p\in(0,1]$, then
\begin{align*}
[W]_{A_p(\mathbb R^n,\mathbb C^m)}
\geq[W]_{A_{p,\infty}(\mathbb R^n,\mathbb C^m)}
\geq1.
\end{align*}
On the other hand, if $p\in(1,\infty)$, then, from H\"older's inequality, it follows that,
for any cube $Q\subset\mathbb{R}^n$ and $x\in Q$,
\begin{align*}
\left\|W^{\frac{1}{p}}(x)\right\|
&\leq\fint_Q\left\|W^{\frac{1}{p}}(x)W^{-\frac{1}{p}}(y)\right\|
\left\|W^{\frac{1}{p}}(y)\right\|\,dy\\
&\leq\left[\fint_Q\left\|W^{\frac{1}{p}}(x)W^{-\frac{1}{p}}(y)\right\|^{p'}\,dy\right]^{\frac{1}{p'}}
\left[\fint_Q\left\|W^{\frac{1}{p}}(y)\right\|^p\,dy\right]^{\frac{1}{p}}
\end{align*}
and hence
\begin{align*}
\fint_Q\left\|W^{\frac{1}{p}}(x)\right\|^p\,dx
\leq[W]_{A_p(\mathbb R^n,\mathbb C^m)}\fint_Q\left\|W^{\frac{1}{p}}(y)\right\|^p\,dy.
\end{align*}
This further implies that $[W]_{A_p(\mathbb R^n,\mathbb C^m)}\geq1$,
which completes the proof of (ii) and hence Proposition \ref{constant}.
\end{proof}

Using Lemma \ref{2.14} and the reverse H\"older inequality of scalar weights,
we obtain the following reverse H\"older inequality of $A_{p,\infty}$-matrix weights.
The case of $A_p$-matrix weights is well known, and the present proof is analogous; both results are based on the facts that the inequality holds for $A_{p,\infty}^{\mathrm{sc}}$-matrix weights (cf. \cite[Lemma 2]{MRR22}), and that both $A_p$ and $A_{p,\infty}$ are contained in $A_{p,\infty}^{\mathrm{sc}}$.

\begin{proposition}\label{reverse Holder matrix}
Let $p\in(0,\infty)$, $W\in A_{p,\infty}(\mathbb R^n,\mathbb C^m)$, and $M\in M_m(\mathbb{C})$.
Then, for any
\begin{align}\label{rr}
r\in\left[1,1+\frac{1}{2^{n+1}[W]_{A_{p,\infty}(\mathbb R^n,\mathbb C^m)}^{\mathrm{sc}}-1}\right]
\end{align}
and any cube $Q\subset\mathbb{R}^n$,
\begin{align*}
\fint_Q\left\|W^{\frac{1}{p}}(x)M\right\|^{pr}\,dx
\leq2\left[\fint_Q\left\|W^{\frac{1}{p}}(x)M\right\|^{p}\,dx\right]^r.
\end{align*}
\end{proposition}

\begin{proof}
If $M$ is a zero matrix, then the present lemma is obvious,
so it is enough to consider the nonzero matrix $M$.
By Lemma \ref{2.14}, we find that the scalar weight
\begin{align*}
w_M:=\left\|W^{\frac{1}{p}}M\right\|^p
\end{align*}
satisfies $w_M\in A_\infty(\mathbb R^n)$ and $[w_M]_{A_\infty(\mathbb R^n)}^*\leq[W]_{A_{p,\infty}(\mathbb R^n,\mathbb C^m)}^{\mathrm{sc}}$.
From this and Lemma \ref{reverse Holder scalar}, we infer that,
for any $r$ in \eqref{rr} and any cube $Q\subset\mathbb{R}^n$,
\begin{align*}
\fint_Q[w_M(x)]^r\,dx
\leq2\left[\fint_Qw_M(x)\,dx\right]^r.
\end{align*}
This finishes the proof of Proposition \ref{reverse Holder matrix}.
\end{proof}

The following corollary is also useful.

\begin{corollary}\label{8}
Let $p\in(0,\infty)$, $W\in A_{p,\infty}$,
and $\{A_Q\}_{Q\in\mathscr{Q}}$ be a sequence of
reducing operators of order $p$ for $W$.
Then the following statements hold.
\begin{enumerate}[{\rm(i)}]
\item There exists a positive constant $C$,
depending only on $m$ and $p$, such that, for any $r$ as in \eqref{rr},
\begin{align*}
\sup_{\mathrm{cube}\,Q}\left[\fint_Q
\left\|W^{\frac{1}{p}}(x)A_Q^{-1}\right\|^{pr}\,dx\right]^{\frac{1}{r}}
\leq C.
\end{align*}
\item There exists a positive constant $C$,
depending only on $m$, $p$, and $[W]_{A_{p,\infty}}$,
such that, for any $r$ as in \eqref{rr},
\begin{align*}
\sup_{Q\in\mathscr{Q}}\left[\fint_Q\sup_{R\in\mathscr{Q},\,x\in R\subset Q}
\left\|W^{\frac{1}{p}}(x)A_R^{-1}\right\|^{pr}\,dx\right]^{\frac{1}{r}}
\leq C.
\end{align*}
\end{enumerate}
\end{corollary}

\begin{proof}
From Proposition \ref{reverse Holder matrix} and Lemma \ref{reduceM}, we deduce that,
for any $r$ in \eqref{rr}
and any cube $Q\subset\mathbb{R}^n$,
\begin{align*}
\left[\fint_Q\left\|W^{\frac{1}{p}}(x)A_Q^{-1}\right\|^{pr}\,dx\right]^{\frac{1}{r}}
<2\fint_Q\left\|W^{\frac{1}{p}}(x)A_Q^{-1}\right\|^{p}\,dx
\sim\left\|A_QA_Q^{-1}\right\|^{p}
=1.
\end{align*}
This finishes the proof of (i).

Applying Lemma \ref{8 prepare} and an argument similar to that used in
the proof of \cite[Lemma A.32]{byy} (see also \cite[Lemma 3.6]{im19}), we obtain (ii).
This finishes the proof of Corollary \ref{8}.
\end{proof}

We conclude this section by indicating an application of Corollary \ref{8}
to the boundedness of certain pointwise multipliers
acting on certain subspaces of $L^p\ell^q$,
the space of sequences $\{f_j\}_{j\in\mathbb Z}$
of measurable functions on $\mathbb R^n$ such that
\begin{equation*}
\left\|\{f_j\}_{j\in\mathbb Z}\right\|_{L^p\ell^q}
:=\left\|\left(\sum_{j\in\mathbb Z}|f_j|^q\right)^{\frac1q}
\right\|_{L^p}<\infty,
\end{equation*}
with usual modification if $q=\infty$.

For any $j\in\mathbb{Z}$, let $\mathscr Q_j$ be the collection of
dyadic cubes of edge length $2^{-j}$ and,
for any nonnegative measurable function $f$ on $\mathbb R^n$
or any $f\in L^1_{\mathrm{loc}}$, let
\begin{equation}\label{Ej}
E_j (f):=\sum_{Q\in\mathscr{Q}_j}
\left[\fint_Q f(x)\,dx\right]\mathbf{1}_Q
\end{equation}
be the related conditional expectation.
Repeating an argument used in the proof of \cite[Corollary 3.8]{fr21}
with \cite[Lemmas 3.2 and 3.3]{fr21} replaced by Corollary \ref{8},
we obtain the following conclusion; we omit the details.

\begin{corollary}\label{46x}
Let $p\in(0,\infty)$, $q\in(0,\infty]$, $W\in A_{p,\infty}$,
and $\{A_Q\}_{Q\in\mathscr{Q}}$ be a sequence of
reducing operators of order $p$ for $W$.
For any $j\in\mathbb{Z}$, let
\begin{align*}
\gamma_j:=\sum_{Q\in\mathscr{Q}_j}
\left\|W^{\frac{1}{p}}A_Q^{-1}\right\|\mathbf{1}_Q.
\end{align*}
Then there exists a positive constant $C$ such that,
for any sequence $\{f_j\}_{j\in\mathbb{Z}}$ of
nonnegative measurable functions on $\mathbb R^n$
or $\{f_j\}_{j\in\mathbb{Z}}\subset L^1_{\mathrm{loc}}$,
$$
\left\|\left\{\gamma_jE_j\left(f_j\right)\right\}_{j\in\mathbb Z}\right\|_{L^p\ell^q}
\leq C\left\|\left\{E_j\left(f_j\right)\right\}_{j\in\mathbb Z}\right\|_{L^p\ell^q},
$$
where $E_j$ for any $j\in\mathbb Z$ is the same as in \eqref{Ej}.
\end{corollary}

Corollary \ref{46x} extends \cite[Corollary 3.8]{fr21}
from $A_p$ to the larger weight class $A_{p,\infty}$.
The said \cite[Corollary 3.8]{fr21} played a key role in the theory of
Triebel--Lizorkin spaces with $A_p$-matrix weights developed in \cite{fr21}
and in its extension to Besov-type and Triebel--Lizorkin-type spaces in \cite{bhyy}.
Corollary \ref{46x} will be similarly used in our follow-up work,
where this theory is extended to the larger class of $A_{p,\infty}$ weights.

\section{Upper and lower dimensions of $A_{p,\infty}$ weights}\label{sec:dimApInfty}

As observed in \cite{bhyy,fr21,ro04},
several results in the theory of matrix-weighted function spaces
depend at the technical level on the possibility of replacing a reducing operator $A_Q$
related to one cube $Q$ by a similar operator $A_R$ for a different cube $R$;
this, in turn, depends on estimates for the norms $\|A_QA_R^{-1}\|$.
Several results may be obtained abstractly, by postulating suitable upper bounds
for the said norms, but it is also of interest to be able to relate such bounds
to the $A_p$ or the $A_{p,\infty}$ conditions of the matrix weight $W$.
In \cite{bhyy}, we introduced the concept of $A_p$-dimension of a matrix weight
and obtained the sharp control of $\|A_QA_R^{-1}\|$ in terms of this quantity
when $W\in A_p$. Our goal in this section is to extend these ideas to
the larger class of $A_{p,\infty}$-matrix weights.

First, we give some useful equivalences among three quantities
on the composition of reducing operators and its related matrix weight.

\begin{proposition}\label{equivalent}
Let $p\in(0,\infty)$, $W\in A_{p,\infty}$,
and $\{A_Q\}_{\mathrm{cube}\,Q}$ be a family of
reducing operators of order $p$ for $W$.
Then, for all cubes $Q,R\subset\mathbb{R}^n$,
\begin{align*}
\left\|A_QA_R^{-1}\right\|^p
&\sim\fint_Q\exp\left(\fint_R\log\left[
\left\|W^{\frac{1}{p}}(x)W^{-\frac{1}{p}}(y)\right\|^p\right]\,dy\right)\,dx\\
&\sim\exp\left(\fint_R\log\left(\fint_Q
\left\|W^{\frac{1}{p}}(x)W^{-\frac{1}{p}}(y)\right\|^p\,dx\right)\,dy\right),
\end{align*}
where the positive equivalence constants
depend only on $m$, $p$, and $[W]_{A_{p,\infty}}$.
\end{proposition}

\begin{proof}
By Lemmas \ref{reduceM}, \ref{exchange}, and \ref{reduceM inv},
we conclude that, for any cubes $Q,R\subset\mathbb{R}^n$,
\begin{align*}
\left\|A_QA_R^{-1}\right\|^p
&\sim\fint_Q\left\|W^{\frac{1}{p}}(x)A_R^{-1}\right\|^p\,dx
=\fint_Q\left\|A_R^{-1}W^{\frac{1}{p}}(x)\right\|^p\,dx\\
&\sim\fint_Q\exp\left(\fint_R\log\left[
\left\|W^{-\frac{1}{p}}(y)W^{\frac{1}{p}}(x)\right\|^p\right]\,dy\right)\,dx\\
&=\fint_Q\exp\left(\fint_R\log\left[
\left\|W^{\frac{1}{p}}(x)W^{-\frac{1}{p}}(y)\right\|^p\right]\,dy\right)\,dx.
\end{align*}
From Lemmas \ref{exchange}, \ref{reduceM inv}, and \ref{reduceM}, we infer that,
for any cubes $Q,R\subset\mathbb{R}^n$,
\begin{align*}
\left\|A_QA_R^{-1}\right\|^p
&=\left\|A_R^{-1}A_Q\right\|^p
\sim\exp\left(\fint_R\log\left[\left\|W^{-\frac{1}{p}}(y)A_Q\right\|^p\right]\,dy\right)\\
&=\exp\left(\fint_R\log\left[\left\|A_QW^{-\frac{1}{p}}(y)\right\|^p\right]\,dy\right)\\
&\sim\exp\left(\fint_R\log\left(\fint_Q
\left\|W^{\frac{1}{p}}(x)W^{-\frac{1}{p}}(y)\right\|^p\,dx\right)\,dy\right).
\end{align*}
These finish the proof of Proposition \ref{equivalent}.
\end{proof}

To obtain a sharp estimate on $\|A_QA_R^{-1}\|$,
we introduce the following concepts of lower and upper dimensions.

\begin{definition}\label{Ap dim}
Let $p\in(0,\infty)$ and $d\in\mathbb{R}$.
A matrix weight $W$ is said to have \emph{$A_{p,\infty}$-lower dimension $d$},
denoted by $W\in\mathbb{D}_{p,\infty,d}^{\mathrm{lower}}(\mathbb{R}^n,\mathbb{C}^m)$,
if there exists a positive constant $C$ such that,
for any cubes $Q,R\subset\mathbb{R}^n$ with $Q\subset R$,
\begin{align*}
\exp\left(\fint_R\log\left(
\fint_Q\left\|W^{\frac{1}{p}}(x)W^{-\frac{1}{p}}(y)\right\|^p\,dx\right)\,dy\right)
\leq C\left[\frac{\ell(R)}{\ell(Q)}\right]^d.
\end{align*}
A matrix weight $W$ is said to have \emph{$A_{p,\infty}$-upper dimension $d$},
denoted by $W\in\mathbb{D}_{p,\infty,d}^{\mathrm{upper}}(\mathbb{R}^n,\mathbb{C}^m)$,
if there exists a positive constant $C$ such that,
for any $\lambda\in[1,\infty)$ and any cube $Q\subset\mathbb{R}^n$,
\begin{align*}
\exp\left(\fint_Q\log\left(
\fint_{\lambda Q}\left\|W^{\frac{1}{p}}(x)W^{-\frac{1}{p}}(y)\right\|^p\,dx\right)\,dy\right)
\leq C\lambda^d.
\end{align*}
\end{definition}

In what follows, if there exists no confusion,
we denote $\mathbb{D}_{p,\infty,d}^{\mathrm{lower}}(\mathbb{R}^n,\mathbb{C}^m)$ simply by
$\mathbb{D}_{p,\infty,d}^{\mathrm{lower}}$ and
$\mathbb{D}_{p,\infty,d}^{\mathrm{upper}}(\mathbb{R}^n,\mathbb{C}^m)$ simply by
$\mathbb{D}_{p,\infty,d}^{\mathrm{upper}}$.
We have the following basic properties on $A_{p,\infty}$-lower dimensions.

\begin{proposition}\label{Ap lower dim prop}
Let $p\in(0,\infty)$. Then the following statements hold.
\begin{enumerate}[\rm(i)]
\item For any $d\in(-\infty,0)$, one has
$\mathbb{D}_{p,\infty,d}^{\mathrm{lower}}=\emptyset$.
\item $\bigcup_{d\in[0,n)}\mathbb{D}_{p,\infty,d}^{\mathrm{lower}}
=A_{p,\infty}$.
\item For any $d\in[n,\infty)$, one has
$\mathbb{D}_{p,\infty,d}^{\mathrm{lower}}=A_{p,\infty}$.
\item For any $d_1,d_2\in[0,n]$
 with $d_1<d_2$, one has
$\mathbb{D}_{p,\infty,d_1}^{\mathrm{lower}}
\subsetneqq\mathbb{D}_{p,\infty,d_2}^{\mathrm{lower}}$.
\end{enumerate}
\end{proposition}

\begin{proof}
We first prove (i). Let $d\in(-\infty,0)$. Assume there exists
$W\in\mathbb{D}_{p,\infty,d}^{\mathrm{lower}}$.
By Definition \ref{MatrixWeight}(iii), we conclude that
\begin{align*}
\left\|W^{\frac{1}{p}}(\cdot)A_{Q_{0,\mathbf{0}}}^{-1}\right\|^p
\in L^1_{\mathrm{loc}}(\mathbb{R}^n),
\end{align*}
which, together with Lebesgue's differentiation theorem
(see, for instance, \cite[Corollary 2.1.16]{g14c}),
further implies that there exist $x_0\in Q_{0,\mathbf{0}}$
such that $W^{\frac1p}(x_0)$ is invertible
and a sequence of cubes $\{Q_k\}_{k=1}^\infty\subset Q_{0,\mathbf{0}}$ such that
the center of every $Q_k$ is $x_0$ and $\lim_{k\to\infty}\ell(Q_k)=0$ and such that
\begin{align*}
\left\|W^{\frac{1}{p}}(x_0)A_{Q_{0,\mathbf{0}}}^{-1}\right\|^p
=\lim_{k\to\infty}\fint_{Q_k}
\left\|W^{\frac{1}{p}}(x)A_{Q_{0,\mathbf{0}}}^{-1}\right\|^p\,dx.
\end{align*}
By this, Lemmas \ref{reduceM} and \ref{equivalent},
and the assumption that $W\in\mathbb{D}_{p,\infty,d}^{\mathrm{lower}}$, we find that
\begin{align*}
\left\|W^{\frac{1}{p}}(x_0)A_{Q_{0,\mathbf{0}}}^{-1}\right\|^p
&\sim\lim_{k\to\infty}\left\|A_{Q_k}A_{Q_{0,\mathbf{0}}}^{-1}\right\|^p\\
&\sim\lim_{k\to\infty}\exp\left(\fint_{Q_{0,\mathbf{0}}}\log\left(
\fint_{Q_k}\left\|W^{\frac{1}{p}}(x)W^{-\frac{1}{p}}(y)\right\|^p\,dx\right)\,dy\right)\\
&\lesssim\lim_{k\to\infty}\left[\frac{\ell(Q_{0,\mathbf{0}})}{\ell(Q_k)}\right]^{d}
=0
\end{align*}
and hence all entries of $W(x_0)$ are $0$,
which contradicts the invertibility of $W^{\frac1p}(x_0)$.
Thus, $\mathbb{D}_{p,\infty,d}^{\mathrm{lower}}=\emptyset$.
This finishes the proof of (i).

Next, we show (ii).
By Definition \ref{Ap dim} with $Q=R$, we obtain, for any $d\in[0,n)$,
$\mathbb{D}_{p,\infty,d}^{\mathrm{lower}}\subset A_{p,\infty}$ and hence
\begin{align}\label{Ap lower dim prop2}
\bigcup_{d\in[0,n)}\mathbb{D}_{p,\infty,d}^{\mathrm{lower}}\subset A_{p,\infty}.
\end{align}
On the other hand, let $W\in A_{p,\infty}$.
From H\"older's inequality and Proposition \ref{reverse Holder matrix} with $r\in(1,\infty)$ as in \eqref{rr},
we deduce that, for any cubes $Q,R\subset\mathbb{R}^n$ with $Q\subset R$,
\begin{align*}
&\exp\left(\fint_R\log\left(
\fint_Q\left\|W^{\frac{1}{p}}(x)W^{-\frac{1}{p}}(y)\right\|^p\,dx\right)\,dy\right)\\
&\quad\leq\exp\left(\fint_R\log\left(\left[
\fint_Q\left\|W^{\frac{1}{p}}(x)W^{-\frac{1}{p}}(y)\right\|^{pr}\,dx
\right]^{\frac{1}{r}}\right)\,dy\right)\\
&\quad\leq\left[\frac{\ell(R)}{\ell(Q)}\right]^{\frac{n}{r}}
\exp\left(\fint_R\log\left(\left[
\fint_R\left\|W^{\frac{1}{p}}(x)W^{-\frac{1}{p}}(y)\right\|^{pr}\,dx
\right]^{\frac{1}{r}}\right)\,dy\right)\\
&\quad\leq2^{\frac{1}{r}}\left[\frac{\ell(R)}{\ell(Q)}\right]^{\frac{n}{r}}
\exp\left(\fint_R\log\left(
\fint_R\left\|W^{\frac{1}{p}}(x)W^{-\frac{1}{p}}(y)\right\|^p\,dx
\right)\,dy\right)\\
&\quad
\leq2^{\frac{1}{r}}[W]_{A_{p,\infty}}\left[\frac{\ell(R)}{\ell(Q)}\right]^{\frac{n}{r}}
\end{align*}
and hence $W$ has $A_{p,\infty}$-lower dimension $\frac{n}{r}\in(0,n)$.
Therefore,
\begin{align*}
A_{p,\infty}\subset\bigcup_{d\in[0,n)}\mathbb{D}_{p,\infty,d}^{\mathrm{lower}}.
\end{align*}
This, combined with \eqref{Ap lower dim prop2}, further implies (ii).

Now, we prove (iii). Let $d\in[n,\infty)$.
By Definition \ref{Ap dim} with $Q=R$, we obtain
$\mathbb{D}_{p,\infty,d}^{\mathrm{lower}}\subset A_{p,\infty}$.
On the other hand, for any $W\in A_{p,\infty}$,
from the definition of $[W]_{A_{p,\infty}}$, we infer that,
for any cubes $Q,R\subset\mathbb{R}^n$ with $Q\subset R$,
\begin{align*}
&\exp\left(\fint_R\log\left(
\fint_Q\left\|W^{\frac{1}{p}}(x)W^{-\frac{1}{p}}(y)\right\|^p\,dx\right)\,dy\right)\\
&\quad\leq\left[\frac{\ell(R)}{\ell(Q)}\right]^n\exp\left(\fint_R\log\left(
\fint_R\left\|W^{\frac{1}{p}}(x)W^{-\frac{1}{p}}(y)\right\|^p\,dx\right)\,dy\right)\\
&\quad\leq[W]_{A_{p,\infty}}\left[\frac{\ell(R)}{\ell(Q)}\right]^d
\end{align*}
and hence $W\in\mathbb{D}_{p,\infty,d}^{\mathrm{lower}}$.
Therefore, $\mathbb{D}_{p,\infty,d}^{\mathrm{lower}}=A_{p,\infty}$,
which completes the proof of (iii).

The inclusion ``$\subset$'' in Claim (iv) follows directly from Definition \ref{Ap dim}.
That this is a strict inclusion ``$\subsetneqq$'' follows from the examples given in Lemma \ref{example Ap dim} below.
This completes the proof of Proposition \ref{Ap lower dim prop}.
\end{proof}

Now, we can give the following basic properties of $A_{p,\infty}$-upper dimensions.

\begin{proposition}\label{Ap upper dim prop}
Let $p\in(0,\infty)$. Then the following statements hold.
\begin{enumerate}[\rm(i)]
\item For any $d\in(-\infty,0)$,
$\mathbb{D}_{p,\infty,d}^{\mathrm{upper}}=\emptyset$.
\item $\bigcup_{d\in[0,\infty)}\mathbb{D}_{p,\infty,d}^{\mathrm{upper}}=A_{p,\infty}$.
\item For any $d_1,d_2\in[0,\infty)$ with $d_1<d_2$,
$\mathbb{D}_{p,\infty,d_1}^{\mathrm{upper}}
\subsetneqq\mathbb{D}_{p,\infty,d_2}^{\mathrm{upper}}$.
\item For any $p\in(0,1]$, $A_p\subset \mathbb{D}_{p,\infty,0}^{\mathrm{upper}}$.
\end{enumerate}
\end{proposition}

Although we mostly deal with $A_{p,\infty}$ weights here, the last claim is intentionally stated for $A_p$ weights only.

\begin{proof}
We first prove (i). Let $d\in(-\infty,0)$.
Assume that there exists
$W\in\mathbb{D}_{p,\infty,d}^{\mathrm{upper}}$.
By Lemma \ref{reduceM inv}, we conclude that
\begin{align*}
\log\left[\left\|W^{-\frac{1}{p}}(\cdot)A_{Q_{0,\mathbf{0}}}\right\|^p\right]
\in L^1_{\mathrm{loc}}(\mathbb{R}^n),
\end{align*}
which, together with Lebesgue's differentiation theorem
(see, for instance, \cite[Corollary 2.1.16]{g14c}),
further implies that there exist $x_0\in Q_{0,\mathbf{0}}$
such that $W^{\frac1p}(x_0)$ is invertible
and a sequence of cubes $\{Q_k\}_{k=1}^\infty\subset Q_{0,\mathbf{0}}$ such
that the center of every $Q_k$ for any $k\in\mathbb N$ is $x_0$
and $\lim_{k\to\infty}\ell(Q_k)=0$ and such that
\begin{align*}
\left\|W^{-\frac{1}{p}}(x_0)A_{Q_{0,\mathbf{0}}}\right\|^p
=\lim_{k\to\infty}\exp\left(\fint_{Q_k}
\log\left[\left\|W^{-\frac{1}{p}}(y)A_{Q_{0,\mathbf{0}}}\right\|^p\right]\,dy\right).
\end{align*}
By this, Lemmas \ref{exchange} and \ref{reduceM},
and $W\in\mathbb{D}_{p,\infty,d}^{\mathrm{upper}}$, we find that
\begin{align*}
&\left\|W^{-\frac{1}{p}}(x_0)A_{Q_{0,\mathbf{0}}}\right\|^p\\
&\quad=\lim_{k\to\infty}\exp\left(\fint_{Q_k}
\log\left[\left\|A_{Q_{0,\mathbf{0}}}W^{-\frac{1}{p}}(y)\right\|^p\right]\,dy\right)\\
&\quad\sim\lim_{k\to\infty}\exp\left(\fint_{Q_k}\log\left(
\fint_{Q_{0,\mathbf{0}}}\left\|W^{\frac{1}{p}}(x)W^{-\frac{1}{p}}(y)\right\|^p\,dx\right)\,dy\right)\\
&\quad\lesssim\lim_{k\to\infty}\exp\left(\fint_{Q_k}\log\left(
\fint_{Q_{x_0}}\left\|W^{\frac{1}{p}}(x)W^{-\frac{1}{p}}(y)\right\|^p\,dx\right)\,dy\right)\\
&\quad\lesssim\lim_{k\to\infty}[\ell(Q_k)]^{-d}
=0,
\end{align*}
where $Q_{x_0}$ is centered at $x_0$, has edge length $\sim 1$, and $Q_{0,\mathbf{0}}\subset Q_{x_0}$.
Hence all entries of $W^{-\frac1p}(x_0)$ are $0$,
which contradicts the fact that this is the inverse matrix of $W^{\frac1p}(x_0)$.
Thus, $\mathbb{D}_{p,\infty,d}^{\mathrm{upper}}=\emptyset$.
This finishes the proof of (i).

Now, we show (ii).
By Definition \ref{Ap dim} with $\lambda=1$, we obtain, for any $d\in[0,\infty)$,
$\mathbb{D}_{p,\infty,d}^{\mathrm{upper}}\subset A_{p,\infty}$ and hence
\begin{align}\label{Ap upper dim prop2}
\bigcup_{d\in[0,\infty)}\mathbb{D}_{p,\infty,d}^{\mathrm{upper}}\subset A_{p,\infty}.
\end{align}
It remains to prove $A_{p,\infty}\subset
\bigcup_{d\in[0,\infty)}\mathbb{D}_{p,\infty,d}^{\mathrm{upper}}$.
Let $W\in A_{p,\infty}$ and $M\in(0,\infty)$, where $M$ is determined later.
For any cube $Q\subset\mathbb R^n$ and any $x\in\mathbb R^n$, let
\begin{align*}
h_Q(x):=\begin{cases}
M&\text{if }x\in Q,\\
1&\text{if }x\in\mathbb{R}^n\setminus Q.
\end{cases}
\end{align*}
For any cube $Q\subset\mathbb{R}^n$,
by Proposition \ref{2.13} with $Q$ and $H$ replaced,
respectively, by $2Q$ and $h_QA_Q^{-1}$, we obtain
\begin{align*}
[W]_{A_{p,\infty}}
&\geq\left[\fint_{2Q}\left\|W^{\frac{1}{p}}(x)A_Q^{-1}\right\|^p[h_Q(x)]^p\,dx\right]^{-1}\\
&\quad\times\exp\left(\fint_{2Q}\log\left(
\fint_{2Q}\left\|W^{\frac{1}{p}}(x)A_Q^{-1}\right\|^p[h_Q(y)]^p\,dx\right)\,dy\right)\\
&=:I(Q)\times II(Q).
\end{align*}
From Lemma \ref{reduceM} and the definition of $h_Q$, it follows that,
for any cube $Q\subset\mathbb R^n$,
\begin{align*}
[I(Q)]^{-1}
&=\fint_{2Q}\left\|W^{\frac{1}{p}}(x)A_Q^{-1}\right\|^p\,dx
+(M^p-1)2^{-n}\fint_Q\left\|W^{\frac{1}{p}}(x)A_Q^{-1}\right\|^p\,dx\\
&\sim\left\|A_{2Q}A_Q^{-1}\right\|^p+M^p-1
\end{align*}
and
\begin{align*}
II(Q)
&=\left[\fint_{2Q}\left\|W^{\frac{1}{p}}(x)A_Q^{-1}\right\|^p\,dx\right]
\exp\left(\fint_{2Q}\log\left\{[h_Q(y)]^p\right\}\,dy\right)\\
&\sim\left\|A_{2Q}A_Q^{-1}\right\|^p\exp\left(\frac{\log(M^p)}{2^n}\right).
\end{align*}
Thus, there exists a constant $C\in(0,1]$, depending only on $m$ and $p$, such that,
for any cube $Q\subset\mathbb R^n$,
\begin{align*}
[W]_{A_{p,\infty}}
&\geq C\frac{\|A_{2Q}A_Q^{-1}\|^p}{\|A_{2Q}A_Q^{-1}\|^p+M^p-1}
\exp\left(\frac{\log(M^p)}{2^n}\right).
\end{align*}
Choosing $M\in(0,\infty)$ such that $CM^{\frac{p}{2^n}}=2[W]_{A_{p,\infty}}$, then,
for any cube $Q\subset\mathbb R^n$, one has
\begin{align*}
\left\|A_{2Q}A_Q^{-1}\right\|^p
\leq M^p-1<M^p
=\left(\frac{2[W]_{A_{p,\infty}}}C\right)^{2^n}
\end{align*}
and hence
\begin{align*}
&\exp\left(\fint_Q\log\left(\fint_{\lambda Q}
\left\|W^{\frac{1}{p}}(x)W^{-\frac{1}{p}}(y)\right\|^p\,dx\right)\,dy\right)\\
&\quad\lesssim\exp\left(\fint_Q\log\left(
\fint_{2^{\lceil\log_2\lambda\rceil}Q}
\left\|W^{\frac{1}{p}}(x)W^{-\frac{1}{p}}(y)\right\|^p\,dx
\right)\,dy\right)\\
&\quad\sim\left\|A_{2^{\lceil\log_2\lambda\rceil}Q}A_Q^{-1}\right\|^{p}
\leq\left(\frac{2[W]_{A_{p,\infty}}}C\right)^{2^n\lceil\log_2\lambda\rceil}\\
&\quad\leq\left(\frac{2[W]_{A_{p,\infty}}}C\right)^{2^n(1+\log_2\lambda)}
\sim\lambda^{2^n(1-\log_2 C+\log_2[W]_{A_{p,\infty}})},
\end{align*}
which further implies that $W$ has $A_{p,\infty}$-upper dimension
$2^n(1-\log_2 C+\log_2[W]_{A_{p,\infty}})\in[0,\infty)$.
Therefore,
\begin{align*}
A_{p,\infty}
\subset\bigcup_{d\in[0,\infty)}
\mathbb{D}_{p,\infty,d}^{\mathrm{upper}}.
\end{align*}
From this and \eqref{Ap upper dim prop2}, it follows that
\begin{align*}
\bigcup_{d\in[0,\infty)}\mathbb{D}_{p,\infty,d}^{\mathrm{upper}}
=A_{p,\infty},
\end{align*}
which completes the proof of (ii).

Claim (iv) and the inclusion ``$\subset$'' in (iii) follow directly from Definitions \ref{Ap dim} and \ref{def ap}.
That there is a strict inclusion ``$\subsetneqq$'' in (iii) follows from the examples given in Lemma \ref{example Ap dim} below.
This finishes the proof of Proposition \ref{Ap upper dim prop}.
\end{proof}

For any matrix weight $W\in A_{p,\infty}$, let
$$d_{p,\infty}^{\mathrm{lower}}(W):=
\inf\{d\in[0,n):\ W\text{ has }A_{p,\infty}\text{-lower dimension }d\}$$
and
$$d_{p,\infty}^{\mathrm{upper}}(W):=
\inf\{d\in[0,\infty):\ W\text{ has }A_{p,\infty}\text{-upper dimension }d\}.$$
Let
\begin{align*}
[\![d_{p,\infty}^{\mathrm{lower}}(W),n):=\begin{cases}
[d_{p,\infty}^{\mathrm{lower}}(W),n)
&\text{if }d_{p,\infty}^{\mathrm{lower}}(W)
\text{ is }A_{p,\infty}\text{-lower dimension of }W,\\
(d_{p,\infty}^{\mathrm{lower}}(W),n)
&\text{otherwise}
\end{cases}
\end{align*}
and
\begin{align*}
[\![d_{p,\infty}^{\mathrm{upper}}(W),\infty):=\begin{cases}
[d_{p,\infty}^{\mathrm{upper}}(W),\infty)
&\text{if }d_{p,\infty}^{\mathrm{upper}}(W)
\text{ is }A_{p,\infty}\text{-upper dimension of }W,\\
(d_{p,\infty}^{\mathrm{upper}}(W),\infty)
&\text{otherwise}.
\end{cases}
\end{align*}
Finally, using the concept of $A_{p,\infty}$-dimensions,
we obtain the following \emph{sharp} estimate.
For the sharpness, see Lemma \ref{146x} further below.

\begin{proposition}\label{sharp estimate}
Let $p\in(0,\infty)$, $W\in A_{p,\infty}$,
and $\{A_Q\}_{\mathrm{cube}\,Q}$ be a family of
reducing operators of order $p$ for $W$.
Let $d_1\in[\![d_{p,\infty}^{\mathrm{lower}}(W),\infty)$ and
$d_2\in[\![d_{p,\infty}^{\mathrm{upper}}(W),\infty)$.
Then there exists a positive constant $C$ such that,
for any cubes $Q,R\subset\mathbb{R}^n$,
\begin{equation*}
\left\|A_QA_R^{-1}\right\|^p
\leq C\max\left\{\left[\frac{\ell(R)}{\ell(Q)}\right]^{d_1},
\left[\frac{\ell(Q)}{\ell(R)}\right]^{d_2}\right\}
\left[1+\frac{|c_Q-c_R|}{\ell(Q)\vee\ell(R)}\right]^{d_1+d_2}.
\end{equation*}
\end{proposition}

\begin{proof}
Let us first consider the special case when $Q\subset R$.
In this case, by Lemma \ref{equivalent}
and Definition \ref{Ap dim}, we find that
\begin{align}\label{Q subset R}
\left\|A_QA_R^{-1}\right\|^p
\sim\exp\left(\fint_R\log\left(\fint_Q
\left\|W^{\frac{1}{p}}(x)W^{-\frac{1}{p}}(y)\right\|^p\,dx\right)
\,dy\right)
\lesssim\left[\frac{\ell(R)}{\ell(Q)}\right]^{d_1}.
\end{align}

Next, we consider the special case where $R\subset Q$.
In this case, using some geometrical observations,
we find that there exists $\lambda\in[1,\infty)$ such that
$\lambda\sim\ell(Q)/\ell(R)$ and $Q\subset\lambda R$.
These, together with Lemma \ref{reduceM} and Proposition \ref{equivalent},
further imply that
\begin{align}\label{R subset Q}
\left\|A_QA_R^{-1}\right\|^p
&\sim\exp\left(\fint_R\log\left(\fint_Q
\left\|W^{\frac{1}{p}}(x)W^{-\frac{1}{p}}(y)\right\|^p\,dx\right)\,dy\right)\\
&\lesssim\exp\left(\fint_R\log\left(\fint_{\lambda R}
\left\|W^{\frac{1}{p}}(x)W^{-\frac{1}{p}}(y)\right\|^p\,dx\right)\,dy\right)\notag\\
&\lesssim\lambda^{d_2}
\sim\left[\frac{\ell(Q)}{\ell(R)}\right]^{d_2}.\notag
\end{align}

In the general case, we choose a third cube $S$ such that $Q\cup R\subset S$.
This clearly can be achieved with
$\ell(S)\sim\ell(Q)+\ell(R)+|c_Q-c_R|$ by some geometrical observations.
From this, \eqref{R subset Q}, and \eqref{Q subset R}, we deduce that
\begin{align*}
\left\|A_QA_R^{-1}\right\|^p
&\leq\left\|A_QA_S^{-1}\right\|^p\left\|A_SA_R^{-1}\right\|^p
\lesssim\left[\frac{\ell(S)}{\ell(Q)}\right]^{d_1}
\left[\frac{\ell(S)}{\ell(R)}\right]^{d_2}\\
&=\left[\frac{\ell(Q)\vee\ell(R)}{\ell(Q)}\right]^{d_1}
\left[\frac{\ell(S)}{\ell(Q)\vee\ell(R)}\right]^{d_1+d_2}
\left[\frac{\ell(Q)\vee\ell(R)}{\ell(R)}\right]^{d_2}\\
&\sim\max\left\{\left[\frac{\ell(R)}{\ell(Q)}\right]^{d_1},
\left[\frac{\ell(Q)}{\ell(R)}\right]^{d_2}\right\}
\left[1+\frac{|c_Q-c_R|}{\ell(Q)\vee\ell(R)}\right]^{d_1+d_2}.
\end{align*}
This finishes the proof of Proposition \ref{sharp estimate}.
\end{proof}

The following lemma is useful;
since its proof is just some basic calculations, we omit the details.

\begin{lemma}\label{33y}
For any cubes $Q,R\subset\mathbb{R}^n$,
any $x,x'\in Q$, and any $y,y'\in R$,
$$
1+\frac{|x-y|}{\ell(Q)\vee\ell(R)}
\sim1+\frac{|x'-y'|}{\ell(Q)\vee\ell(R)},
$$
where the positive equivalence constants depend only on $n$.
\end{lemma}

Next, we recall the concept of dyadic cubes.
For any $j\in\mathbb{Z}$ and $k:=(k_1,\ldots,k_n)\in\mathbb{Z}^n$, let
$$
Q_{j,k}:=\prod_{i=1}^n2^{-j}[k_i,k_i+1),
$$
$$
\mathscr{Q}:=\{Q_{j,k}:\ j\in\mathbb{Z},\ k\in\mathbb{Z}^n\},
$$
$$
\mathscr{Q}_+:=\{Q_{j,k}:\ j\in\mathbb{Z}_+,\ k\in\mathbb{Z}^n\},
$$
and
\begin{align*}
\mathscr{Q}_{j}:=\{Q_{j,k}:\ k\in\mathbb{Z}^n\}.
\end{align*}
For any $Q:=Q_{j,k}\in\mathscr{Q}$, we let $j_Q:=j$ and $x_Q:=2^{-j}k$.

Using Proposition \ref{sharp estimate} and Lemma \ref{33y},
we directly obtain the following estimates; we omit the details.

\begin{lemma}\label{sharp}
Let $p\in(0,\infty)$, $W\in A_{p,\infty}$,
and $\{A_Q\}_{\mathrm{cube}\,Q}$ be a family of
reducing operators of order $p$ for $W$.
Let $d_1\in[\![d_{p,\infty}^{\mathrm{lower}}(W),n)$ and
$d_2\in[\![d_{p,\infty}^{\mathrm{upper}}(W),\infty)$.
Then the following three statements hold.
\begin{enumerate}[{\rm(i)}]
\item There exists a positive constant $C$ such that,
for any $Q,R\in\mathscr{Q}$,
\begin{align*}
\left\|A_QA_R^{-1}\right\|^p
\leq C\max\left\{\left[\frac{\ell(R)}{\ell(Q)}\right]^{d_1},
\left[\frac{\ell(Q)}{\ell(R)}\right]^{d_2}\right\}
\left[1+\frac{|x_Q-x_R|}{\ell(Q)\vee\ell(R)}\right]^{d_1+d_2}.
\end{align*}
\item There exists a positive constant $C$ such that,
for any $j\in\mathbb Z$ and $Q,R\in\mathscr{Q}_j$,
\begin{align*}
\left\|A_QA_R^{-1}\right\|^p
\leq C\left(1+2^j|x_Q-x_R|\right)^{d_1+d_2}.
\end{align*}
\item There exists a positive constant $C$ such that,
for any cube $Q,R\subset\mathbb{R}^n$ with $Q\cap R\neq\emptyset$,
\begin{align*}
\left\|A_QA_R^{-1}\right\|^p
\leq C\max\left\{\left[\frac{\ell(R)}{\ell(Q)}\right]^{d_1},
\left[\frac{\ell(Q)}{\ell(R)}\right]^{d_2}\right\}.
\end{align*}
\end{enumerate}
\end{lemma}

\section{Examples and sharpness}\label{sec:examples}

To show the sharpness of Proposition \ref{sharp estimate},
corresponding to Definition \ref{Ap dim},
we introduce the concepts of lower and upper dimensions of scalar weights.

\begin{definition}
Let $p\in(0,\infty)$ and $d\in\mathbb{R}$.
A scalar weight $w$ is said to have \emph{$A_\infty$-lower dimension $d$},
denoted by $w\in\mathbb{D}_{\infty,d}^{\mathrm{lower}}(\mathbb{R}^n)$,
if there exists a positive constant $C$ such that,
for any cubes $Q,R\subset\mathbb{R}^n$ with $Q\subset R$,
\begin{align*}
\fint_Qw(x)\,dx\exp\left(\fint_R\log\left\{[w(x)]^{-1}\right\}\,dx\right)
\leq C\left[\frac{\ell(R)}{\ell(Q)}\right]^d.
\end{align*}
A scalar weight $w$ is said to have \emph{$A_\infty$-upper dimension $d$},
denoted by $w\in\mathbb{D}_{\infty,d}^{\mathrm{upper}}(\mathbb{R}^n)$,
if there exists a positive constant $C$ such that,
for any $\lambda\in[1,\infty)$ and any cube $Q\subset\mathbb{R}^n$,
\begin{align*}
\fint_{\lambda Q}w(x)\,dx\exp\left(\fint_Q\log\left\{[w(x)]^{-1}\right\}\,dx\right)
\leq C\lambda^d.
\end{align*}
\end{definition}

The following lemma gives the relations between scalar and matrix weights,
which follows immediately from their definitions; we omit the details.

\begin{lemma}\label{w vs W 2}
Let $p\in(0,\infty)$ and $d\in[0,\infty)$.
Let $w$ be a scalar weight and $W:=wI_m$,
where $I_m$ is identity matrix.
Then the following statements hold.
\begin{enumerate}[\rm(i)]
\item $W\in\mathbb{D}_{p,\infty,d}^{\mathrm{lower}}$ if and only if $w\in\mathbb{D}_{\infty,d}^{\mathrm{lower}}$;
\item $W\in\mathbb{D}_{p,\infty,d}^{\mathrm{upper}}$ if and only if $w\in\mathbb{D}_{\infty,d}^{\mathrm{upper}}$.
\end{enumerate}
\end{lemma}

$A_\infty$-dimensions of scalar weights are analogues of
the doubling and the reverse doubling conditions,
which reveal the intrinsic property of $A_\infty$-lower and $A_\infty$-upper dimensions.

\begin{proposition}\label{2.34}
Let $p\in(0,\infty)$, $d\in\mathbb{R}$, and scalar weight $w\in A_\infty$.
Then
\begin{enumerate}[\rm(i)]
\item $w\in\mathbb{D}_{\infty,d}^{\mathrm{lower}}(\mathbb{R}^n)$
if and only if there exists a positive constant $C$ such that,
for any $\lambda\in[1,\infty)$ and any cube $Q\subset\mathbb{R}^n$,
\begin{align}\label{2.35}
\frac{w(Q)}{w(\lambda Q)}
\leq C\lambda^{d-n};
\end{align}
\item $w\in\mathbb{D}_{\infty,d}^{\mathrm{upper}}(\mathbb{R}^n)$
if and only if there exists a positive constant $C$ such that,
for any $\lambda\in[1,\infty)$ and any cube $Q\subset\mathbb{R}^n$,
\begin{align*}
\frac{w(\lambda Q)}{w(Q)}
\leq C\lambda^{d+n}.
\end{align*}
\end{enumerate}
\end{proposition}

\begin{proof}
By case $m=1$ of the definition of a reducing operator \eqref{equ_reduce}
and the estimate of its inverse \eqref{reduceM inv equ2}, we obtain
\begin{align}\label{2.35y}
\exp\left(\fint_R\log\left\{[w(x)]^{-1}\right\}\,dx\right)
\sim\frac{|R|}{w(R)},
\end{align}
which directly implies (ii) and the necessity of (i).
Now, we show the sufficiency of (i).
From $w\in A_\infty(\mathbb{R}^n)$, we infer that there exists $p\in(1,\infty)$
such that $w\in A_p(\mathbb{R}^n)$.
By \eqref{2.35y} and \eqref{AQ inv},
we find that, for any cube $R\subset\mathbb{R}^n$,
\begin{align}\label{2.35x}
\exp\left(\fint_R\log\left\{[w(x)]^{-1}\right\}\,dx\right)
\sim\left\{\fint_R[w(x)]^{-\frac{p'}p}\,dx\right\}^{\frac p{p'}},
\end{align}
where $\frac1p+\frac1{p'}=1$.
For any cubes $Q,R\subset\mathbb{R}^n$ with $Q\subset R$,
we have $R\subset\lambda Q$, where $\lambda\sim\ell(R)/\ell(Q)$,
which, combined with \eqref{2.35x} and case $m=1$ of \eqref{2.35}, further implies that
\begin{align*}
&\fint_Qw(x)\,dx\exp\left(\fint_R\log\left\{[w(x)]^{-1}\right\}\,dx\right)\\
&\quad\sim\fint_Qw(x)\,dx\left\{\fint_R[w(x)]^{-\frac{p'}p}\,dx\right\}^{\frac p{p'}}\\
&\quad\lesssim\fint_Qw(x)\,dx\left\{\fint_{\lambda Q}[w(x)]^{-\frac{p'}p}\,dx\right\}^{\frac p{p'}}\\
&\quad\sim\fint_Qw(x)\,dx\exp\left(\fint_{\lambda Q}\log\left\{[w(x)]^{-1}\right\}\,dx\right)\\
&\quad\sim\frac{|\lambda Q|}{|Q|}\frac{w(Q)}{w(\lambda Q)}
\lesssim\lambda^d
\sim\left[\frac{\ell(R)}{\ell(Q)}\right]^d
\end{align*}
and hence $w\in\mathbb{D}_{\infty,d}^{\mathrm{lower}}(\mathbb{R}^n)$.
This finishes the proof of Proposition \ref{2.34}.
\end{proof}

For any scalar weight $w$, let
\begin{align*}
d_{\infty}^{\mathrm{lower}}(w):=
\inf\{d\in[0,n):\ w\text{ has }A_{\infty}\text{-lower dimension }d\}
\end{align*}
and
\begin{align*}
d_{\infty}^{\mathrm{upper}}(w):=
\inf\{d\in[0,\infty):\ w\text{ has }A_{\infty}\text{-upper dimension }d\}.
\end{align*}

\begin{lemma}\label{example Ap dim 2}
Let $a\in(-n,\infty)$ and $b\in\mathbb{R}$.
Define the scalar weight $w_{a,b}(x):=|x|^a[\log(2+|x|)]^b$ for any $x\in\mathbb{R}^n$.
Then $w_{a,b}\in A_\infty(\mathbb{R}^n)$ and
\begin{enumerate}[\rm(i)]
\item for any $d\in(a_-,\infty)$, there exists a positive constant $C$ such that,
for any $\lambda\in[1,\infty)$ and any cube $Q\subset\mathbb R^n$,
\begin{align}\label{lower}
\frac{w_{a,b}(Q)}{w_{a,b}(\lambda Q)}\leq C\lambda^{d-n};
\end{align}
\item for any $d\in(-\infty,a_-)$, the estimate \eqref{lower} does not hold;
\item for $d=a_-$, the estimate \eqref{lower} holds
if and only if $a\in(0,\infty)$ or $b\in[0,\infty)$;
\item for any $d\in(a_+,\infty)$, there exists a positive constant $C$ such that,
for any $\lambda\in[1,\infty)$ and any cube $Q\subset\mathbb R^n$,
\begin{align}\label{upper}
\frac{w_{a,b}(\lambda Q)}{w_{a,b}(Q)}\leq C\lambda^{d+n};
\end{align}
\item for any $d\in(-\infty,a_+)$, the estimate \eqref{upper} does not hold;
\item for $d=a_+$,
the estimate \eqref{upper} holds
if and only if $a\in(-\infty,0)$ or $b\in(-\infty,0]$.
\end{enumerate}
\end{lemma}

\begin{proof}
By \cite[Lemma 2.40 and Proposition 2.37]{bhyy}, we obtain (i)-(iii).
From \cite[Corollary 2.42]{bhyy}, we deduce that,
for any $\lambda\in[1,\infty)$ and any cube $Q\subset\mathbb R^n$,
\begin{align*}
\frac{w_{a,b}(\lambda Q)}{w_{a,b}(Q)}
&\sim\lambda^n\left[\frac{|c_Q|+\lambda\ell(Q)}{|c_Q|+\ell(Q)}\right]^a
\left\{\frac{\log[2+|c_Q|+\lambda\ell(Q)]}{\log[2+|c_Q|+\ell(Q)]}\right\}^b\\
&\sim\lambda^{2n}\frac{w_{-a,-b}(Q)}{w_{-a,-b}(\lambda Q)}.
\end{align*}
This, together with (i)-(iii), further implies (iv)-(vi),
which complete the proof of Lemma \ref{example Ap dim 2}.
\end{proof}

Applying Proposition \ref{2.34} and Lemma \ref{example Ap dim 2},
we obtain the following conclusion; we omit the details.

\begin{lemma}\label{example Ap dim}
Let $a\in(-n,\infty)$ and $b\in\mathbb{R}$.
Let $w_{a,b}$ be the same as in Lemma \ref{example Ap dim 2}.
Then $w_{a,b}\in A_\infty(\mathbb{R}^n)$ and
\begin{enumerate}[\rm(i)]
\item $d_{\infty}^{\mathrm{lower}}(w_{a,b})=a_-$
is $A_\infty$-lower dimension of $w_{a,b}$
if and only if $a\in(0,\infty)$ or $b\in[0,\infty)$;
\item $d_{\infty}^{\mathrm{upper}}(w_{a,b})=a_+$
is $A_\infty$-upper dimension of $w_{a,b}$
if and only if $a\in(-\infty,0)$ or $b\in(-\infty,0]$.
\end{enumerate}
\end{lemma}

When $|x|\lesssim1$, we have $w_{a,b}(x)\sim|x|^a$ and
the disturbance term $[\log(2+|x|)]^b$ has no effect.
To solve this problem, we need the following variant of Lemma \ref{example Ap dim}.

\begin{lemma}\label{exa}
Let $a\in(-n,\infty)$ and $b\in\mathbb{R}$.
Let scalar weight $\widetilde w_{a,b}(x):=|x|^a[\log(2+|x|^{-1})]^b$ for any $x\in\mathbb{R}^n$.
Then $\widetilde w_{a,b}\in A_\infty(\mathbb{R}^n)$ and
\begin{enumerate}[\rm(i)]
\item $d_{\infty}^{\mathrm{lower}}(\widetilde w_{a,b})=a_-$
is $A_\infty$-lower dimension of $\widetilde w_{a,b}$
if and only if $a\in(0,\infty)$ or $b\in(-\infty,0]$;
\item $d_{\infty}^{\mathrm{upper}}(\widetilde w_{a,b})=a_+$
is $A_\infty$-upper dimension of $\widetilde w_{a,b}$
if and only if $a\in(-\infty,0)$ or $b\in[0,\infty)$.
\end{enumerate}
\end{lemma}

To prove Lemma \ref{exa}, we need a technical lemma.
Applying an argument similar to that used in the proof of \cite[Lemma 2.41]{bhyy},
we obtain the following estimate; we omit the details.

\begin{lemma}\label{135}
Let $a\in(-n,\infty)$, $b\in\mathbb{R}$, and
$\widetilde w_{a,b}$ be the same as in Lemma \ref{exa}.
Then, for any $x_0\in\mathbb{R}^n$ and $r\in(0,\infty)$,
\begin{align*}
\fint_{B(x_0,r)}\widetilde w_{a,b}(x)\,dx
\sim(|x_0|+r)^a\left\{\log\left[2+(|x_0|+r)^{-1}\right]\right\}^b,
\end{align*}
where the positive equivalence constants depend only on $n$, $a$, and $b$.
\end{lemma}

Lemma \ref{135} remains true if we replace balls $B$ therein by cubes $Q$;
we omit the details.

\begin{corollary}\label{135x}
Let $a\in(-n,\infty)$, $b\in\mathbb{R}$, and
$\widetilde w_{a,b}$ be the same as in Lemma \ref{exa}.
Then, for any cube $Q\subset\mathbb{R}^n$,
$$
\fint_Q\widetilde w_{a,b}(x)\,dx
\sim[|c_Q|+\ell(Q)]^a\left(\log\left\{2+[|c_Q|+\ell(Q)]^{-1}\right\}\right)^b,
$$
where the positive equivalence constants depend only on $n$, $a$, and $b$.
\end{corollary}

It is the time for us to prove Lemma \ref{exa}.

\begin{proof}[Proof of Lemma \ref{exa}]
We first show (i).
Let $a\in(-n,\infty)$ and $b\in\mathbb{R}$.
By \cite[Lemma 2.3(v)]{fs97}, we find that $\widetilde w_{a,b}\in A_\infty(\mathbb R^n)$.
This, combined with Proposition \ref{2.34}(i),
further implies that, to prove (i),
we need only to show that
\begin{enumerate}[\rm(a)]
\item for any $d\in(a_-,\infty)$, there exists a positive constant $C$ such that,
for any $\lambda\in[1,\infty)$ and any cube $Q\subset\mathbb R^n$,
\begin{align}\label{lower 2}
\frac{\widetilde w_{a,b}(Q)}{\widetilde w_{a,b}(\lambda Q)}\leq C\lambda^{d-n};
\end{align}
\item for any $d\in(-\infty,a_-)$, \eqref{lower 2} does not hold;
\item for $d=a_-$, \eqref{lower 2} holds
if and only if $a\in(0,\infty)$ or $b\in(-\infty,0]$.
\end{enumerate}

Now, we prove (a). Let $d\in(a_-,\infty)$.
By Lemma \ref{135x}, we conclude that,
for any $\lambda\in[1,\infty)$ and any cube $Q\subset\mathbb R^n$,
\begin{equation}\label{137}
\frac{\widetilde w_{a,b}(Q)}{\widetilde w_{a,b}(\lambda Q)}
\sim\lambda^{-n}\left[\frac{|c_Q|+\ell(Q)}{|c_Q|+\lambda\ell(Q)}\right]^a
\left[\frac{\log(2+[|c_Q|+\ell(Q)]^{-1})}{\log(2+[|c_Q|+\lambda\ell(Q)]^{-1})}\right]^b.
\end{equation}
Applying an argument similar to that used in the estimation of \cite[(2.37)]{bhyy},
we obtain, for any $\lambda\in[1,\infty)$,
$$
\frac{\widetilde w_{a,b}(Q)}{\widetilde w_{a,b}(\lambda Q)}
\lesssim\lambda^{(a_-)-n}(1+\log_2\lambda)^{b_+}
\lesssim\lambda^{d-n}.
$$
This finishes the proof of (a).

Next, we show (b).
By \eqref{137}, we obtain, for any $\lambda\in[1,\infty)$
and any cube $Q\subset\mathbb R^n$ with $c_Q=\mathbf{0}$ and $\ell(Q)=1$,
$$
\frac{\widetilde w_{a,b}(Q)}{\widetilde w_{a,b}(\lambda Q)}
\gtrsim\lambda^{-a-n}\left[\frac{\log3}{\log(2+\lambda^{-1})}\right]^b;
$$
for any $\lambda\in[1,\infty)$
and any cube $Q\subset\mathbb R^n$ with $|c_Q|=\lambda$ and $\ell(Q)=1$,
$$
\frac{\widetilde w_{a,b}(Q)}{\widetilde w_{a,b}(\lambda Q)}
\gtrsim\lambda^{-n}\left[\frac{\lambda+1}{2\lambda}\right]^a
\left[\frac{\log(2+(\lambda+1)^{-1})}{\log(2+(2\lambda)^{-1})}\right]^b
\sim\lambda^{-n}
$$
and hence \eqref{lower 2} does not hold if $d\in(-\infty,a_-)$.
This finishes the proof of (b).

Finally, applying an argument similar to that used in the proof of \cite[Lemma 2.40(ii)]{bhyy},
we obtain (c). This finishes the proof of (i).

Now, we show (ii).
By Lemma \ref{135x}, we conclude that,
for any $\lambda\in[1,\infty)$ and any cube $Q\subset\mathbb R^n$,
\begin{align*}
\frac{\widetilde w_{a,b}(\lambda Q)}{\widetilde w_{a,b}(Q)}
&\sim\lambda^{n}\left[\frac{|c_Q|+\lambda\ell(Q)}{|c_Q|+\ell(Q)}\right]^a
\left[\frac{\log(2+[|c_Q|+\lambda\ell(Q)]^{-1})}{\log(2+[|c_Q|+\ell(Q)]^{-1})}\right]^b\\
&\sim\lambda^{2n}\frac{\widetilde w_{-a,-b}(Q)}{\widetilde w_{-a,-b}(\lambda Q)}.
\end{align*}
Applying this and (a)-(c), we obtain
\begin{enumerate}
\item[\rm(d)] for any $d\in(a_+,\infty)$, there exists a positive constant $C$ such that,
for any $\lambda\in[1,\infty)$ and any cube $Q\subset\mathbb R^n$,
\begin{align}\label{upper 2}
\frac{\widetilde w_{a,b}(\lambda Q)}{\widetilde w_{a,b}(Q)}\leq C\lambda^{d+n};
\end{align}
\item[\rm(e)] for any $d\in(-\infty,a_+)$, \eqref{upper 2} does not hold;
\item[\rm(f)] for $d=a_+$, \eqref{upper 2} holds
if and only if $a\in(-\infty,0)$ or $b\in[0,\infty)$.
\end{enumerate}
These, together with Proposition \ref{2.34}(ii), further imply (ii),
which then completes the proof of Lemma \ref{exa}.
\end{proof}

Finally, we prove that the estimate of $\|A_Q A_R^{-1}\|$ in terms of the upper and lower $A_{p,\infty}$-dimensions, as given in Proposition \ref{sharp estimate}, is \emph{sharp}.

\begin{lemma}\label{146x}
Let $d_1\in[0,n)$, $d_2\in[0,\infty)$, and $a,b,c\in\mathbb R$.
\begin{enumerate}[\rm(i)]
\item If, for any $W\in A_{p,\infty}$ satisfying that
$d_{p,\infty}^{\mathrm{lower}}(W)=d_1$,
$d_{p,\infty}^{\mathrm{upper}}(W)=d_2$,
$W\in\mathbb D_{p,\infty,d_1}^{\mathrm{lower}}$,
and $W\in\mathbb D_{p,\infty,d_2}^{\mathrm{upper}}$,
there exists a positive constant $C$ such that,
for any cubes $Q,R\subset\mathbb{R}^n$,
\begin{align}\label{2.26x}
\left\|A_QA_R^{-1}\right\|^p
\leq C\max\left\{\left[\frac{\ell(R)}{\ell(Q)}\right]^a,
\left[\frac{\ell(Q)}{\ell(R)}\right]^b\right\}
\left[1+\frac{|c_Q-c_R|}{\ell(Q)\vee\ell(R)}\right]^c,
\end{align}
where $\{A_Q\}_{\mathrm{cube}\,Q}$ is a family of
reducing operators of order $p$ for $W$, then
\begin{align*}
a\in[d_1,\infty),\
b\in[d_2,\infty),\text{ and }
c\in[d_1+d_2,\infty).
\end{align*}
\item If, for any $W\in A_{p,\infty}$ satisfying that
$d_{p,\infty}^{\mathrm{lower}}(W)=d_1$,
$d_{p,\infty}^{\mathrm{upper}}(W)=d_2$,
$W\notin\mathbb D_{p,\infty,d_1}^{\mathrm{lower}}$,
and $W\in\mathbb D_{p,\infty,d_2}^{\mathrm{upper}}$,
\eqref{2.26x} holds, then
\begin{align*}
a\in(d_1,\infty),\
b\in[d_2,\infty),\text{ and }
c\in(d_1+d_2,\infty).
\end{align*}
\item If, for any $W\in A_{p,\infty}$ satisfying that
$d_{p,\infty}^{\mathrm{lower}}(W)=d_1$,
$d_{p,\infty}^{\mathrm{upper}}(W)=d_2$,
$W\in\mathbb D_{p,\infty,d_1}^{\mathrm{lower}}$,
and $W\notin\mathbb D_{p,\infty,d_2}^{\mathrm{upper}}$,
\eqref{2.26x} holds, then
\begin{align*}
a\in[d_1,\infty),\
b\in(d_2,\infty),\text{ and }
c\in(d_1+d_2,\infty).
\end{align*}
\item If, for any $W\in A_{p,\infty}$ satisfying that
$d_{p,\infty}^{\mathrm{lower}}(W)=d_1$,
$d_{p,\infty}^{\mathrm{upper}}(W)=d_2$,
$W\notin\mathbb D_{p,\infty,d_1}^{\mathrm{lower}}$,
and $W\notin\mathbb D_{p,\infty,d_2}^{\mathrm{upper}}$,
\eqref{2.26x} holds, then
\begin{align}\label{146x equ4}
a\in(d_1,\infty),\
b\in(d_2,\infty),\text{ and }
c\in(d_1+d_2,\infty).
\end{align}
\end{enumerate}
\end{lemma}

\begin{proof}
Due to similarity, we only show (iv).
Let $x_0:=(1,0,\ldots,0)$ and matrix weight $W:=wI_m$,
where $I_m$ is identity matrix and, for any $x\in\mathbb R^n$,
$$
w(x):=w_1(x)[w_2(x)]^{-1}:=\widetilde w_{-d_1,1}(x)[\widetilde w_{-d_2,1}(x-x_0)]^{-1}.
$$
Now, we prove that $W$ satisfies all assumptions of (iv).
We first show that $d_\infty^{\mathrm{lower}}(w)=d_1$ and
$w\notin\mathbb D_{\infty,d_1}^{\mathrm{lower}}(\mathbb R^n)$.
To this end, from Proposition \ref{2.34}(i), it follows that
we need only to prove that
\begin{enumerate}[\rm(a)]
\item for any $d\in(d_1,\infty)$, there exists a positive constant $C$ such that,
for any $\lambda\in[1,\infty)$ and any cube $Q\subset\mathbb R^n$,
\begin{align}\label{lower 3}
\frac{w(Q)}{w(\lambda Q)}\leq C\lambda^{d-n};
\end{align}
\item for $d=d_1$, \eqref{lower 3} does not hold.
\end{enumerate}

Next, we show (a). Let $d\in(d_1,\infty)$.
By \cite[Lemma 2.3(v)]{fs97}, Lemma \ref{exa}(i),
and Proposition \ref{2.34}(i), we conclude that, for any $i\in\{1,2\}$,
$w_i\in A_1(\mathbb R^n)$ and,
for any $\lambda\in[1,\infty)$ and any cube $Q\subset\mathbb R^n$,
\begin{align}\label{wi}
\fint_Qw_1(x)\,dx\left\|w_1^{-1}\right\|_{L^\infty(\lambda Q)}
&=\frac{|\lambda Q|}{|Q|}\frac{w_1(Q)}{w_1(\lambda Q)}
\fint_{\lambda Q}w_1(x)\,dx\left\|w_1^{-1}\right\|_{L^\infty(\lambda Q)}\\
&\leq\lambda^n\frac{w_1(Q)}{w_1(\lambda Q)}[w_1]_{A_1(\mathbb R^n)}
\lesssim\lambda^d.\notag
\end{align}
From $w_1,w_2\in A_1(\mathbb R^n)$ and the Jones factorization theorem,
we infer that $w\in A_2(\mathbb R^n)\subset A_\infty(\mathbb R^n)$.
This, combined with Lemma \ref{w vs W}, further implies that $W\in A_{p,\infty}$.
By $w\in A_2(\mathbb R^n)$, \eqref{AQ inv}, \eqref{wi},
and the definition of $[w_2]_{A_1(\mathbb R^n)}$, we find that,
for any $\lambda\in[1,\infty)$ and any cube $Q\subset\mathbb R^n$,
\begin{align*}
\frac{w(Q)}{w(\lambda Q)}
&\sim\lambda^{-n}\fint_Qw(x)\,dx\fint_{\lambda Q}[w(x)]^{-1}\,dx\\
&\leq\lambda^{-n}\fint_Qw_1(x)\,dx\left\|w_1^{-1}\right\|_{L^\infty(\lambda Q)}
\fint_{\lambda Q}w_2(x)\,dx\left\|w_2^{-1}\right\|_{L^\infty(Q)}\\
&\lesssim\lambda^{d-n}[w_2]_{A_1(\mathbb R^n)}.
\end{align*}
This finishes the proof of (a).

Now, we show (b).
By Corollary \ref{135x}, we obtain,
for any cube $Q\subset\mathbb R^n$ with $c_Q=\mathbf{0}$ and $\ell(Q)<\frac12$,
\begin{align*}
w(Q)
\sim|Q|\fint_Qw_1(x)\,dx
\sim[\ell(Q)]^{n-d_1}\log[2+\ell(Q)^{-1}].
\end{align*}
Thus, for any $\lambda\in(0,1)$ and
any cube $Q\subset\mathbb R^n$ with $c_Q=\mathbf{0}$ and $\ell(Q)=\frac13$,
\begin{align*}
\frac{w(\lambda Q)}{w(Q)}
\sim\lambda^{n-d_1}\frac{\log(2+3\lambda^{-1})}{\log5},
\end{align*}
where
\begin{align*}
\lim_{\lambda\to0^+}\frac{\log(2+3\lambda^{-1})}{\log5}=\infty
\end{align*}
and hence \eqref{lower 3} does not hold if $d=d_1$.
This finishes the proof of (b) and hence $d_\infty^{\mathrm{lower}}(w)=d_1$ and
$w\notin\mathbb D_{\infty,d_1}^{\mathrm{lower}}(\mathbb R^n)$.
From these and Lemma \ref{w vs W 2},
it follows that $d_{p,\infty}^{\mathrm{lower}}(W)=d_1$ and
$W\notin\mathbb D_{p,\infty,d_1}^{\mathrm{lower}}$.
Applying an argument similar to that used in the proofs of both
$d_{p,\infty}^{\mathrm{lower}}(W)=d_1$ and
$W\notin\mathbb D_{p,\infty,d_1}^{\mathrm{lower}}$, we obtain $d_{p,\infty}^{\mathrm{upper}}(W)=d_2$ and
$W\notin\mathbb D_{p,\infty,d_2}^{\mathrm{upper}}$.
Therefore, $W$ satisfies all assumptions of (iv).

Applying an argument similar to that used in the proof of \cite[Lemma 2.47]{bhyy},
we obtain \eqref{146x equ4}.
This finishes the proof of Lemma \ref{146x}.
\end{proof}



\newpage

\noindent Fan Bu

\medskip

\noindent Laboratory of Mathematics and Complex Systems (Ministry of Education of China),
School of Mathematical Sciences, Beijing Normal University, Beijing 100875, The People's Republic of China

\smallskip

\noindent{\it E-mail:} \texttt{fanbu@mail.bnu.edu.cn}

\bigskip

\noindent Tuomas Hyt\"onen (Corresponding author)

\medskip

\noindent Department of Mathematics and Statistics,
University of Helsinki, (Pietari Kalmin katu 5), P.O.
Box 68, 00014 Helsinki, Finland

\smallskip

\noindent{\it E-mail:} \texttt{tuomas.hytonen@helsinki.fi}

\bigskip

\noindent Dachun Yang and Wen Yuan

\medskip

\noindent Laboratory of Mathematics and Complex Systems (Ministry of Education of China),
School of Mathematical Sciences, Beijing Normal University, Beijing 100875, The People's Republic of China

\smallskip

\noindent{\it E-mails:} \texttt{dcyang@bnu.edu.cn} (D. Yang)

\noindent\phantom{{\it E-mails:} }\texttt{wenyuan@bnu.edu.cn} (W. Yuan)

\end{document}